\documentclass[a4paper,11pt,reqno,noindent]{amsart}
\usepackage[centertags]{amsmath}
\usepackage{amsfonts,amssymb,amsthm} 
\usepackage{hyperref}

 \hypersetup{
     pdfmenubar=false,        
     pdfnewwindow=true,      
     colorlinks=false,       
     linkcolor=blue,          
     citecolor=blue,        
     filecolor=magenta,      
     urlcolor=cyan           
 }
\usepackage{graphicx} 

\usepackage[english]{babel}
\usepackage{newlfont}
\usepackage{color}
\usepackage[body={15cm,21.5cm},centering]{geometry} 
\usepackage{fancyhdr}
\usepackage{esint}
\usepackage{enumerate}

\usepackage[active]{srcltx}

\newtheorem{theorem}{Theorem}[section]
\newtheorem{lemma}[theorem]{Lemma}

\newtheorem{remark}[theorem]{Remark}
\newtheorem{definition}[theorem]{Definition}
\newtheorem{notation}[theorem]{Notation}

\theoremstyle{definition}

\newcommand{\loc}{\mathrm{loc}}

\newcommand{\diam}{\mathrm{diam}\,}

\newcommand{\Lip}{\mathrm{Lip}\,}

\newcommand{\dist}{\operatorname{dist}}
\newcommand{\supp}{\operatorname{supp}}

\newcommand{\eq}{\stackrel{.}{=}}
\newcommand{\Kappa}{\mathrm{K}}
\newcommand{\e}{\varepsilon}

\renewcommand{\S}{{\mathcal S}}
\newcommand{\R}{\mathbb{R}}

\newcommand{\Z}{\mathbb{Z}}
\newcommand{\N}{\mathbb{N}}

\renewcommand{\d}{\mathrm{d}}
\renewcommand{\L}{\mathcal{L}}

\newcommand\sgn{\mathrm{sgn}}

\renewcommand{\H}{\mathcal{H}}
\newcommand{\Sb}{\mathbf{S}}
\renewcommand{\i}{\mathrm{i}}

\newcommand\wto{\rightharpoonup}

\newcommand{\degd}{{\mathrm{deg}}^\partial}
    
\renewcommand{\div}{\mathrm{div}\,} 
\newcommand{\db}{\mathrm{dim}_{\mathrm{box}}\,}            
\newcommand{\bd}{\mathrm{dim}_{\mathrm{box}}\,}

\title{Integrability of the Brouwer degree for irregular arguments}
\date{\today}
\author[H. Olbermann] {Heiner Olbermann}
\date{\today}
\address[Heiner Olbermann]{Universit\"at Leipzig, Germany}
\email{heiner.olbermann@math.uni-leipzig.de}

\begin{document}

\maketitle
\begin{abstract}
We prove that the Brouwer degree $\deg(u,U,\cdot)$ for a function $u\in C^{0,\alpha}( U;\R^n)$ is in $L^p(\R^n)$ if $1\leq p<\frac{n\alpha}d$, where $U\subset \R^n$ is open and bounded and $d$ is the box dimension of $\partial U$. This is supplemented by a theorem showing that $u_j\to u$ in $C^{0,\alpha}(U;\R^n)$ implies $\deg(u_j,U,\cdot)\to \deg(u,U,\cdot)$ in $L^p(\R^n)$ for the  parameter regime $1\leq p<\frac{n\alpha}d$, while there exist convergent sequences $u_j\to u$ in $C^{0,\alpha}(U;\R^n)$ such that $\|\deg(u_j,U,\cdot)\|_{L^p}\to \infty$ for the opposite regime $p>\frac{n\alpha}d$.
\end{abstract}

\section{Introduction}
The Brouwer degree is a very useful object in
nonlinear analysis, in particular in problems with a geometric background. 
One notable example of its use is the $C^{1,\alpha}$ isometric
immersion problem (see \cite{conti2012h}), where the integrability properties of
the degree are crucial. \\
For a Lipschitz function $u:U\to\R^n$, where $U\subset\R^n$ is open and
bounded, the integrability of the Brouwer degree is as good as one could hope,
namely, there is the classical ``change of variables''-type formula
\begin{equation}
\int_U \varphi(u(x))\det D u(x) \d x=\int_{\R^n}\varphi(z)\deg(u,U,z)\d z\label{eq:17}
\end{equation}
for all $\varphi\in L^1(\R^n)$ (see e.g.~\cite{MR0257325}). However, when
the regularity of $u$ is worse -- only $C^{0,\beta}$ for some $0<\beta<1$ -- it
is much less clear how to deal with integrals as the one on the right hand side
above. To obtain information about such integrals, we will use the fact that
$\deg(u,U,y)\d y$ is an exact form (see e.g.~\cite{MR1373430}) and try to
apply Stokes' Theorem to write it as a boundary integral. This in turn needs
some regularity of the boundary $\partial U$. Usually, one needs $U$ to be a set of finite perimeter to
be able to apply Stokes' Theorem. In \cite{MR1119189}, it has been shown that  if the integrand is
smooth enough, then Stokes' Theorem may also be applied to sets with rougher
boundary. The first aim of the present paper is to adapt these ideas to the case
of the Brouwer degree and show that  $\deg(u,U,\cdot)$ is integrable if  $u$ is smooth enough in terms of H\"older regularity, and  $\partial U$ is smooth
enough in terms of its box dimension. We will show that there is a trade-off
between these two types of regularity. 
\begin{theorem}
\label{thm:main1}
Let $0<\alpha<1$  and  $n-1<d<n$ such that $n\alpha>d$, and let 
 $U\subset \R^n$ be open and bounded with $\db \partial U=d$. 
 Furthermore , let $u\in C^{0,\alpha}(U;\R^n)$. Then $\deg(u,U,\cdot)\in L^p(\R^n)$ for all
  $1\leq p<\frac{n\alpha}{d}$, and for all $p\in (1,\frac{n\alpha}{d})$, there exists a constant $C=C(n, U,\alpha,d,p)$ such that
\begin{equation*}
  \|\deg(u,U,\cdot)\|_{L^p}\leq C \|u\|_{C^{0,\alpha}(U;\R^n)}^{n/p}\,.
\end{equation*}
\end{theorem}
In fact, we will prove this theorem by giving a meaning to the left hand side in
the change of variables formula \eqref{eq:17}, with the regularity of $u$, $U$ as
stated in the theorem. We will show how to make sense of the left hand side for
$u\in
C^{0,\alpha}(U;\R^n)$ and $\varphi\in L^{p'}$ where $p'$ is defined by requiring $p^{-1}+(p')^{-1}=1$. The main idea is to represent $\varphi(u(x))\det Du(x)$
as a sum of Jacobian determinants, interpreted in a weak sense. There are two  crucial tools that will allow
us  to do so. First, we use multi-linear (real) interpolation for a suitable
weak definition of the Jacobian determinant, see Lemma
\ref{lem:detinter}. The statement of this lemma can be viewed as a variant  of
Theorem 3 in the paper
\cite{MR2810795} by Brezis and Nguy\^en, which relies  on an idea by Bourgain, Brezis  and
Mironescu \cite{MR1781527,MR2075883}. Second, we use the following trick: Let $\psi$ be a solution of
$\div\psi=\varphi$. Set $U^i:=(u_1,\dots,u_{i-1},\psi_i\circ
u,u_{i+1},\dots,u_n)$. Then we have
\[
\begin{split}
  \det D U^i(x)\d x=&\d u_1(x)\wedge\dots\wedge\d u_{i-1}(x)\wedge \d (\psi_i\circ
  u)(x)\wedge \dots\wedge \d u_n(x)\\
  =& \partial_i \psi_i(u(x)) \d u_1(x)\wedge\dots\wedge \d u_n(x) \\
=& \partial_i \psi_i(u(x)) \det
  Du (x)\d x\,.
\end{split}
\]
Hence, we get 

\begin{equation}
\begin{split}
  \varphi(u(x))\det Du (x)=&\sum_{i=1}^n\partial_i\psi_i(u(x))\det
  Du(x)\\
  =&\sum_{i=1}^n \det DU^i(x),
\end{split}\label{eq:9}
\end{equation}
which is the sought-for representation as a sum of Jacobian determinants.\\
\\
We have already noted that by the change of variables formula \eqref{eq:17}, the
integrability of the Brouwer degree is closely related to the weakest space for
which we can define the distributional Jacobian determinant $[Ju]$. The question
for the weakest space in which $[Ju]$ can be defined has a long history,
starting with the work of Morrey\cite{MR2492985}, Reshetnyak \cite{revsetnjak1968weak} and Ball \cite{MR0475169}, and with
important contributions by many researchers, see e.g.~\cite{MR0322651,MR1062920,MR1182504,MR1310951,MR1354598,MR1422201,MR1886622}, and
references therein. In the recent article \cite{MR2810795}, this question has
been answered by the use of fractional Sobolev spaces. In this reference, 
$[Ju]$ has been defined as an element of the dual of $C^1$ for $u\in
W^{(n-1)/n,n}$. This result contains most of the previously known ones, such as
the definition of $[Ju]$ for $u\in W^{1,n-1}\cap L^\infty$ or $u\in
W^{1,n^2/(n+1)}$, see \cite{MR0475169}.\\ 
\\
Paralleling  the methods from
\cite{MR2810795}, or using the results from \cite{MR2745198}, one can define $[Ju]$ as an element of $(C^{0,\alpha})^*$
for $u\in C^{0,\alpha}$ and $\alpha>n/(n+1)$. Using this definition, formula \eqref{eq:17} has a well defined meaning for
$\varphi\in C^1$, since then $\varphi\circ u\in C^{0,\alpha}$. Note however that
our treatment using the relation \eqref{eq:9}, which exploits the special
structure of the test function, 
gives meaning to \eqref{eq:17} for a much larger class of test
functions. In particular, if we assume  that $U$ has Lipschitz boundary, then we
will be able to give a well-defined meaning to the left hand side in
\eqref{eq:17} for $u\in C^{0,\alpha}$ and  $\varphi\in L^{p'}$ with
$\alpha/p>(n-1)/n$ (where $p^{-1}+(p')^{-1}=1$), which coincides with the right
hand side.\\
\\
The question whether there exist $\alpha$-H\"older functions whose mapping
degree is not in $L^p$ for $n\alpha<pd$ is not addressed here. Note however that for
$n\alpha<d$, the image of the boundary $u(\partial U)$ has in general
non-vanishing Lebesgue measure, and hence $\deg(u,U,\cdot)$ is not defined on a
set of positive measure (cf.~Lemma \ref{lem:bdrydim}).\\
\\
As a supplement to Theorem \ref{thm:main1}, we show that convergence in
$C^{0,\alpha}$ implies convergence of the associated mapping degrees in $L^p$ if
$n\alpha>pd$, while for the opposite regime $n\alpha<pd$, there exist sequences
that converge to $0$ in $C^{0,\alpha}$ whose mapping degrees diverge in $L^p$.
\begin{theorem}
\label{thm:convcor}
Let $0<\alpha<1$, $n-1<d<n$, $1\leq p<\infty$.
\begin{itemize}
\item[(i)] If $p<\frac{n\alpha}{d}$,  $U\subset \R^n$ is open and bounded
  with $\db \partial U=d$, and  $u_k\in  C^{0,\alpha}(
U;\R^n)$ with $u_k\to u $ in
$C^{0,\alpha}(U;\R^n)$, then $\deg(u_k,U,\cdot)\to \deg(u,U,\cdot)$
in $L^p(\R^n)$.
\item[(ii)] If $p>\frac{n\alpha}{d}$, there exist an open bounded set $U\subset
  \R^n$ with $\db \partial U=d$, and a sequence $u_k\in  C^{0,\alpha}(
U;\R^n)$ with $\deg(u_k,U,\cdot)\in L^p(\R^n)$, $u_k\to 0 $ in
$C^{0,\alpha}(U;\R^n)$ and $\|\deg(u_k,U,\cdot)\|_{L^p}\to\infty$.
\end{itemize}
\end{theorem}

We end this introduction by explaining the plan of the paper. In Section
\ref{sec:preliminaries}, we collect some known methods and theorems that we are
going to need in our proofs. They concern (real) interpolation theory,
self-similar fractals, the Brouwer degree, the Whitney decomposition of an open
subset of $\R^n$,  and the relation between the Whitney decomposition and the
box dimension. In Section \ref{sec:proof-first-thm}, we give the proof of our
 main result, Theorem \ref{thm:main1}.  Section \ref{sec:proof-thm2} is
devoted to the proof of Theorem \ref{thm:convcor}, using several Lemmas whose
proof is given in Section \ref{sec:proof-lemmas-sec4}. 

\subsection*{Notation}
The
symbol for the non-negative integers is $\N=\{0,1,\dots\}$.
The open ball in $\R^n$ with center $x\in\R^n$ and radius $r>0$ will be denoted by $B(x,r)$, while
the open ball in $\R^{n-1}$ with center $x\in\R^{n-1}$ and radius $r$  will be denoted by $B^{n-1}(x,r)$. The standard $n-1$ sphere is
$\Sb^{n-1}=\{x\in\R^n:|x|=1\}$. The canonical orthonormal basis of $\R^n$ is denoted by $(e_1,\dots,e_n)$.
The characteristic function of a set $A\subset \R^n$ is denoted by
$\chi_A$.\\ 
The $n$-dimensional Lebesgue measure is denoted by
$\mathcal L^n$, and the $k$-dimensional Hausdorff measure by $\H^k$. The volume of the unit ball in $m$ dimensions is denoted by
$\omega_m=\pi^{m/2}/\Gamma(m/2+1)$.\\ 
Whenever we want to say that two functions $f,g$ that are defined $\L^n$-almost
everywhere on $\R^n$, agree $\L^n$-almost everywhere, then we write $f\eq
g$. \\
For a Lipschitz
function defined on a set $A\subset \R^n$, its Lipschitz constant is $\Lip
f=\sup_{x,y\in A,x\neq y}|f(x)-f(y)|/|x-y|$. 
For sets $A\subset \R^n$ and functions $f:A\to\R$, we set
\[
[f]_{C^{0,\alpha}(A)}:=\sup_{\substack{x,y\in A\\x\neq
    y}}\frac{|f(x)-f(y)|}{|x-y|^{\alpha}}\,.
\]
If the domain is clear, we often will write $[f]_\alpha\equiv
[f]_{C^{0,\alpha}(A)}$ for short. The corresponding H\"older norm is defined by
\[
\|f\|_{C^{0,\alpha}(A)}=\sup_{x\in A}|f(x)|+[f]_{C^{0,\alpha}(A)}\,.
\]
Let $\Lambda^{p}\R^n$ denote the set of rank $p$ multi-vectors in $\R^n$, i.e.,
the linear space 
\[
\Lambda^p\R^n=\left\{ \sum_{i_1,\dots,i_p\in\{1,\dots,n\}}a_{i_1,\dots,i_p}\d x_{i_1}\wedge\dots \wedge \d x_{i_p} :\,
 a_{i_1,\dots,i_p}\in\R\right\}\,,
\]
With this notation, $p$-forms are functions on $U$ with values in $\Lambda^p\R^n$. We 
make $C^{k,\alpha}(U;\Lambda^p\R^n)$  a normed space by setting
\[
\|a\|_{C^{k,\alpha}(U;\Lambda^p\R^n)}=\sum_{i_1,\dots,i_p}\|a_{i_1,\dots,i_p}\|_{C^{k,\alpha}(U)}
\]
for $a=\sum_{i_1,\dots,i_p}a_{i_1,\dots,i_p}\d x_{i_1}\wedge\dots \wedge \d
x_{i_p} $.

The symbol $C$ will have the
following meaning: A statement such as 
$f\leq C(a,b,\dots) g $
  means that there
exists a numerical constant $C$ that only depends on $a,b,\dots$, such that
$f\leq C g $. The value of $C$ may change from one line to the next.

\subsection*{Acknowledgments}
The author would like to thank Camillo De Lellis and Dominik Inauen for pointing out an error in the proof of Theorem 1.1 that had led to the wrong exponent on the right hand side in the estimate for the $L^p$-norm of the degree. 
Also, he  would like to thank Stefan M\"uller for helpful discussions, in
particular for pointing him to the results on multi-linear interpolation in the
reference \cite{MR2328004}. 
\section{Preliminaries}
\label{sec:preliminaries}
\subsection{Tools from interpolation theory}
\label{sec:tools-from-interp}
We are going to use some standard constructions from real interpolation
theory, due to Lions and Peetre \cite{MR0133693,MR0165343} (see also the
textbook \cite{MR0482275}).  In the following, we give a very short definition
of interpolation spaces via the  trace method \cite{MR0159212}.\\
Let $(E_0,\|\cdot\|_0)$, $(E_1,\|\cdot\|_1)$ be normed spaces. 
We may equip $E_0\cap E_1$ and $E_0+E_1$ with the following norms:
\[
\begin{split}
  \|x\|_{E_0\cap E_1}=&\max\{\|x\|_0,\|x\|_1\}\\
  \|x\|_{E_0+ E_1}=&\inf\{\|x_0\|_0+\|x_1\|_1:\,x_0\in E_0,\,x_1\in
E_1,\,x_0+x_1=x\}
\end{split}
\]
\begin{definition}
\label{def:tracespace}
For $\theta \in (0,1) $ and  $1 \leq p \leq \infty$ we denote by $V(p, \theta,
E_1, E_0)$ the set of all functions $u \in W^{1,p}_{\mathrm{loc}}(\R_+,E_0+
E_1)$ with the following properties:  $u(t)\in E_1$ and $u'(t)\in E_0$ for all $t>0$, and with $u_{*,\theta}(t):=t^\theta u(t)$ and
$u'_{*,\theta}(t):=t^\theta u'(t)$, we have
\[
u_{*,\theta} \in  L^p(\R^+,  \d t/t;E_1) ,\quad u'_{*,\theta} \in  L^p(\R^+, \d t/t;E_0)\,.
\]
We define a norm on $V = V(p, \theta, E_1, E_0)$ by
\[
\|u\|_V := \|u_{*,\theta}\|_{ L^p(\R^+,  \d t /t;E_1)} + \| u'_{*,\theta}\|_{ L^p(\R^+, \d t/t;E_0)}\,.
\]
\end{definition}
It can be shown that those functions are continuous in $t=0$ and we define the real interpolation spaces as follows:
\begin{definition} The real interpolation space $(E_0, E_1)_ {\theta, p}$ 
is defined as set of traces of functions  belonging to $V(p, 1-\theta, E_1, E_0)$ at $t=0$ together with the norm:
\[
\|x\|^{\mathrm{Tr}}_{(\theta, p)} = \inf \{ \|u\|_V : u \in
V(p,1-\theta,E_1,E_0), \, \lim_{t\to 0}u(t) = x \}
\]
\end{definition}

It can be shown
that the  H\"older spaces $C^{0,\alpha}(U)$ are identical to the
real interpolation space
$(C^0(U),C^1(U))_{\alpha,\infty}$, up to
equivalence of norms. 
\subsection{Self-similar fractals}
\label{sec:self-simil-fract}
We recall the construction of self-similar fractals introduced in
\cite{MR625600} (see also \cite{MR867284}). A \emph{similarity} is a map $S:\R^n\to \R^n$
such that $|S(x)-S(y)|=c|x-y|$ for all $x,y\in\R^n$, for some $c>0$. The number $c$ is called the
\emph{ratio} of $S$. For $i=1,\dots,k$,
let $S_i$ be such a similarity, with ratios smaller than 1.
A compact set $K\subset \R^n$ is said to be invariant under $\mathcal
S=\{S_1,\dots,S_k\}$ if
\[
K=\cup_{i=1}^k S_i(K)\,.
\]
In fact, one can show that there exists a unique compact set, the attractor set
of $\mathcal S$, denoted by $K(\mathcal S)$,
that fulfills this property. It consists of the closure of the fixed points
of finite compositions of the similarities. A set constructed in this way is
called \emph{self-similar}.\\
For a given set of similarities $\mathcal S=\{S_1,\dots,S_k\}$,
we define a transformation $S$ on the class of non-empty compact sets by 
\begin{equation}
 S(E)=\cup_{i=1}^k S_i(E)\label{eq:32}
 \end{equation}
and write $S^l$ for the $l$-th iterate of $S$.
For $i_1,\dots,i_l\in\{1,\dots,k\}$ and $E\subset \R^n$, we will use the notation
\[
S_{i_1,\dots,i_l}(E)=S_{i_1}\circ \dots \circ S_{i_l}(E)\,.
\]
With this notation, we have
\[
S^l(E)=\bigcup_{i_1,\dots,i_l=1}^k S_{i_1,\dots,i_l}(E)\,.
\]
A convenient way of defining certain self-similar sets in $\R^2$ (i.e., 
self-similar curves) is by specifying a \emph{generator} for the curve. This is
a sequence of points $\gamma:\{1,\dots,k+1\}\to\R^2$ with $|\gamma(1)|<1$,
$|\gamma(i)-\gamma(i-1)|<1$ for $i=2,\dots,k+1$. 
  The set of
similarities associated to such a generator is given by $\{S_1,\dots,S_k\}$,
where $S_i$ is the orientation preserving similarity that maps $(0,0)$ to
$\gamma(i)$ and $(1,0)$ to $\gamma({i+1})$.
A typical example of a self-similar set constructed from a generator is the Koch
curve, see Figure \ref{fig:koch}.
\begin{figure}[h]
\begin{center}
\includegraphics[width=.7\textwidth]{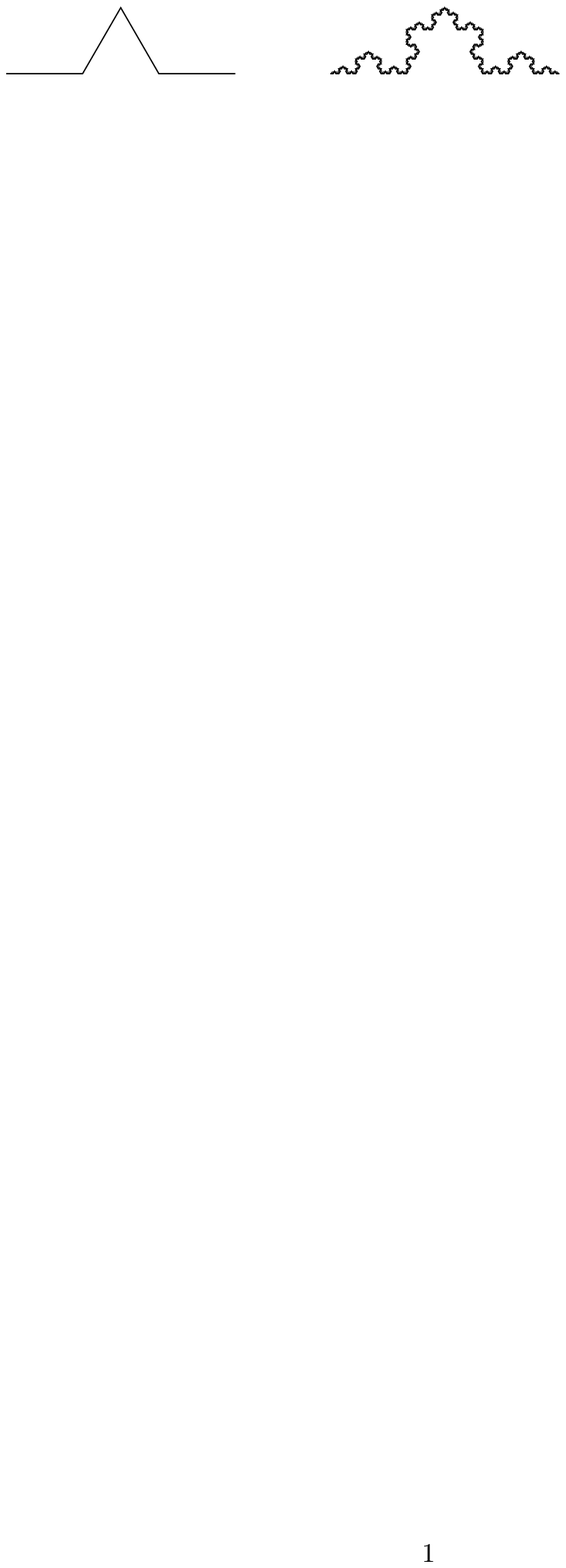}
\caption{The Koch curve (right) and its generator (left).\label{fig:koch}}
\end{center}
\end{figure}
A set of similarities $\mathcal S$ is said to
satisfy the open set condition if there exists a non-empty open set $V\subset \R^n$ such
that
\[
\begin{split}
  S_i(V)\subset V \quad &\text{ for }i=1,\dots,k\\ 
S_i(V)\cap S_j(V)=\emptyset
  \quad&\text{ for }i,j=1,\dots,k,\,i\neq j
\end{split}
\]
The following lemma has been proved in \cite{MR625600,MR3236784}:
\begin{lemma}[Theorem 9.3 in \cite{MR3236784}]
\label{lem:selfsim}
Let $\mathcal S=\{S_1,\dots,S_k\}$ be a set of similarities satisfying the open set condition, and let $r_i$
be the ratio of $S_i$ for $i=1,\dots,k$.
Further, let $d$ be the (unique) real number that satisfies
\[
\sum_{i=1}^k r_i^d=1\,.
\]
Then  the Hausdorff dimension and box dimension of $K(\mathcal S)$ agree and are
equal to $d$.
\end{lemma}

\subsection{Properties of the Brouwer degree}
We  recall the definition and some basic properties of the Brouwer
degree.
For a more thorough
exposition with proofs of the claims made here, see e.g.~\cite{MR1373430}.\\
Let  $U$ be a bounded subset of $\R^n$. Further, let $u\in C^\infty(\overline U;\R^n)$. Assume that $y\in \R^n\setminus u(\partial U)$, and let $\mu$ be a $C^\infty$ $n$-form on $\R^n$ with support in the same connected component of $\R^n\setminus u(\partial U)$ as $y$, such that $\int_{\R^n} \mu= 1$. Then the degree is defined by 
\begin{equation}
\deg(u,U,y)=\int_U u^*(\mu)\,,\label{eq:10}
\end{equation}
where $u^*$ is the pull-back by $u$. It can be shown that this definition is independent of the choice of $\mu$. Further, $\deg(u,U,\cdot)$ is constant on connected components of $\R^n\setminus u(\partial U)$ and integer valued. Moreover,  it is invariant under homotopies, i.e., given $H\in C^\infty([0,1]\times \overline U; \R^n)$ such that $y\not \in H([0,1],\partial U)$, we have
\[
\deg(H(0,\cdot),U,y)=\deg(H(1,\cdot),U,y)\,.
\]
Using these facts, one can go on to define the degree for $u\in C^0(\overline
U;\R^n)$ by approximation.\\
If $ u:\overline U\to\R^n$ is Lipschitz, and $\L^n(\partial U)=0$, then  it follows from \eqref{eq:10} and approximation by smooth functions that
\begin{equation}
\int_{\R^n}\deg(u,U,\cdot)\mu=\int_U  u^*(\mu)
\label{eq:51}
\end{equation}
for any $n$-form $\mu$ on $\R^n$ with coefficients in $L^\infty(\R)$. 
If $\mu$ is an exact form, i.e.,
\[
\mu= \d \omega
\]
for some $n-1$ form $\omega$ on $\R^n$, then
\[
 u^*(\d \omega)= \d \left( u^* \omega\right)\,.
\]
If $U$ has Lipschitz boundary,
this implies, by Stokes' Theorem,

\begin{equation}
\begin{split}
  \int_{\R^n}\deg(u,U, \cdot)\d \omega=&\int_U \d \left( u^* \omega\right)\\
    =&\int_{\partial U}  u^* \omega \,.
  \end{split}\label{eq:57}
  \end{equation}
Assume $\mu$ is a given $n$-form. Since we assume $U$ to be bounded, we may always find some $n-1$-form $\omega$ such that $\d\omega=\mu$ on $\supp \deg(u,U,\cdot)\subset u(U)$, and hence \eqref{eq:57} shows in particular that 
the degree only depends on $u|_{\partial U}$. We will
 write $\deg(u, U, y) =\degd(u,\partial U, y)$. \\
In the proof of Theorem \ref{thm:convcor} (ii), we will use the following lemma:
\begin{lemma}
\label{lem:sep}
Let $U\subset\R^n$ be a bounded Lipschitz domain, and let $u:\bar U\to \R^n$
be Lipschitz. Further, let $V\subset\partial U$ be relatively open in $\partial
U$, and assume there exists $y_0\in\R^n$ such that $u(x)=y_0$ for all
$x\in\tilde\partial V$ (where $\tilde \partial V$ denotes the relative boundary
of $V$ in $\partial U$). Define $u_i:\partial U\to \R^n$, $i=1,2$ by
\[
u_1(x)=\begin{cases}u(x)& \text{ if }x\in V\\ y_0 & \text{ if }x\in \partial U\setminus V\end{cases}\,,\qquad
u_2(x)=\begin{cases} y_0 &
  \text{ if }x\in V\\u(x)& \text{ if }  x\in \partial U\setminus V\end{cases}\,.
\]
Then
\[
\degd(u,\partial U,y)=\degd(u_2,\partial U\setminus V,y)+\degd(u_1,V,y)\quad
\text{for all } y\in \R^n\setminus u(\partial U)\,.
\]
\end{lemma}
\begin{proof}
We will show 
\[
\int_{\R^n}\degd(u,U,\cdot)\mu =\int_{\R^n}\degd(u_1,U,\cdot)\mu+\int_{\R^n}\degd(u_2,U,\cdot)\mu\,
\]
for every $n$-form $\mu$ on $\R^n$ with coefficients in $L^\infty$. Indeed, as  we remarked below \eqref{eq:57}, there exists an $n-1$-form $\omega$ such that $\mu=\d\omega$ on $u(\overline U)$, and hence 
\[
\begin{split}
  \int_{\R^n}\degd(u,U,\cdot)\mu=&\int_{\partial U}u^*\omega\\
  =&\int_{\partial U}u_1^*\omega+\int_{\partial U}u_2^*\omega\\
  =&\int_{\R^n}\degd(u_1,U,\cdot)\mu+\int_{\R^n}\degd(u_2,U,\cdot)\mu\,,
\end{split}
\]
proving the claim of the lemma.
\end{proof}
\subsection{Whitney decomposition and  box dimension}
\label{sec:whitn-decomp-box}
One of our main tools in the proof of Theorem \ref{thm:main1} will be the
Whitney decomposition of an open set $U$.
\begin{lemma}[see e.g.~\cite{MR0290095}, Chapter 1, Theorem 3]
\label{lem:whitneydecomp}
Let $U\subset \R^n$ be open. Then there exists 
 a countable collection $W=\{Q_i:i\in\N\}$ of cubes $Q_i$ with the
following properties:
\begin{itemize}
\item For every $Q\in W$, there exist $k,m_1,\dots,m_n\in\Z$ such that $Q=(m_1
  2^{-k},(m_1+1)2^{-k})\times \dots\times(m_n 2^{-k},(m_n+1)2^{-k})$. For fixed $k$, the union of
  cubes for which this holds for some $m_1,\dots,m_n$ is denoted by $W_k$.
\item $U\subset\cup_{Q\in W} \overline Q$ 
\item The cubes in $W$ are mutually disjoint
\item $\dist(Q,\partial U)\leq \diam Q\leq 4\dist(Q,\partial U)$ for all $Q\in
  W$
\end{itemize}
\end{lemma}
Next, we recall the definition of box dimension, and some of its elementary properties.
\begin{definition}
Let $U\subset \R^n$ be bounded. Let $N_r(U)$ be the number of $n$-dimensional boxes of side length $r$ that
is required to cover $U$. The box dimension  $\db U$  is defined as
\[
\db(U)=\lim_{r\to 0}\frac{\log N_r(U)}{-\log r}\,,
\]
if this limit exists.
\end{definition}
We also define the $\beta$-dimensional Hausdorff-type content for sets $A\subset\R^n$,
\[
H^\beta(A)=\lim_{r\to 0}\left(\inf\{k r^\beta:\cup_{i=1}^{k}B(x_i,r)\supset A\}\right)\,.
\]
If $\db A$ exists, then
\[
\db A= \inf\{\beta :H^{\beta}(A)<\infty\}\,,
\]
see e.g~\cite{MR3236784}, Definition 3.1.\\
\\
In the following lemma, for sets $A\subset \R^n$, we will use the notation
\[
(A)_\e=\{x\in \R^n:\dist (x,A)\leq \e\}\,.
\]
\begin{lemma}
\label{lem:bdrydim}
Let $V\subset\R^n$ be open and bounded, $U\subset\subset V$, $n-1<\db \partial U=d<n$, $0<\alpha<1$ such that $n\alpha>d$,  and $u\in C^{0,\alpha}(V;\R^n)$. Then
\[
\L^n\left(u(\partial U)_\e\right)\to 0\quad\text{ as }\e\to 0\,.
\]
\end{lemma}
\begin{proof}
Set $\tilde \e=\e^{\alpha^{-1}}$.  Choose a finite number of $x_i\in\partial U$, $i=1,\dots,k$, such that
\[
\begin{split}
  \partial U\subset& \bigcup_{i=1}^k B(x_i,\tilde \e)\\
  B(x_i,\tilde \e/5)\cap B(x_j,\tilde \e/5)=&\emptyset \quad\text{ for
  }i,j\in\{1,\dots,k\},i\neq j\,.
\end{split}
\]
Such a collection $\{x_i\}$ exists by the Vitali Covering Lemma.  Choose $d<\bar d<n\alpha$. This choice implies $H^{\bar d}(\partial U)=0$. By choosing $\e$ small enough, we may assume
\[
 k\tilde \e^{\bar d} \leq 1\,, 
\]
Next observe that 
\[
\begin{split}
  u(\partial U)\subset& \bigcup_{i=1}^k B\left(u(x_i),\|u\|_{C^{0,\alpha}}\tilde \e^\alpha\right)\\
  = &\bigcup_{i=1}^k B\left(u(x_i),\|u\|_{C^{0,\alpha}}\e\right)\,.
\end{split}
\]
We set $C^*=\|u\|_{C^{0,\alpha}}+1$ and get
\[
\left(u(\partial U)\right)_\e\subset \bigcup_{i=1}^k B\left(u(x_i),C^*\e\right)\,.
\]
Hence,
\[
\begin{split}
  \L^n\left(\left(u(\partial U)\right)_\e\right)\leq &k \L^n(B(0,1))(C^*\e)^n\\
  \leq & C(u,n) (\tilde \e^{\alpha n-\bar d})k \tilde\e^{\bar d}\\
  \leq & C(u,n) (\tilde \e^{\alpha n-\bar d}) \\
  \to & 0 \text{ as }\e \to 0\,.
\end{split}
\]
This proves the lemma.
\end{proof}

In the proof of Theorem \ref{thm:main1}, we are going to exploit the following relation between the Whitney decomposition
and box dimension:
\begin{theorem}[\cite{MR880256}, Theorem 3.12]
\label{thm:Nkdb}
Let $K\subset \R^n$ be compact, with $\db K=d<n$. Let $W$ be the Whitney
decomposition of $\R^n\setminus K$.
Then $\lim_{k\to\infty}\frac{\log_2 \# W_k}{k}=d$.
 \end{theorem}


\section{Proof of Theorem \ref{thm:main1}}
\label{sec:proof-first-thm}
Recall that $C^1(U;\Lambda^{n-1}\R^n)$ denotes the space of continuously differentiable $n-1$
forms on $U$. For the subspace of closed forms, we introduce the notation
\[
C^1_{\mathrm{cl}}(U;\Lambda^{n-1}\R^n):=\{\omega\in
  C^1(U;\Lambda^{n-1}\R^n):\d\omega=0\}\,.
\]
Now we define two
norms $\|\cdot \|_{X_0^{n-1}}$, $\|\cdot \|_{X^{n-1}_1}$ on the
quotient space $C^1(U;\Lambda^{n-1}\R^n)/C^1_{\mathrm{cl}}(U;\Lambda^{n-1}\R^n)$:
\[
\begin{split}
  \|\omega\|_{X_0^{n-1}}:=&\inf\{\|\omega+\alpha\|_{C^0(U;\Lambda^{n-1}\R^n)}:\alpha\in
  C^1(U;\Lambda^{n-1}\R^n),\d\alpha=0\}\\
  \|\omega\|_{X_1^{n-1}}:=&\|\d\omega\|_{C^0(U;\Lambda^n\R^n)}
\end{split}
\]
Let $X_0^{n-1}$, $X_1^{n-1}$ denote the Banach spaces that one obtains by completion
with respect to the above norms respectively.\\
Next we define a multi-linear operator
\[
\begin{split}
  M:C^1(U;\R^n)\to &C^0(U;\Lambda^{n-1}\R^n)\\
  (u_1,\dots,u_n)\mapsto & \frac1n\sum_{i=1}^n(-1)^{i+1}
u_i\,\d u_1\wedge\dots\wedge\widehat{\d u_i}\wedge\dots\wedge\d u_n\,,
\end{split}
\]
where $\widehat{\d u_i}$ denotes omission of the factor $\d u_i$.
Note that
  \[
  \d M(u_1,\dots,u_n)=\det Du\, \d x_1\wedge\dots \d x_n\equiv \det Du\, \d x\,.
  \]
In the following lemma, let $X_{\theta}$ denote the real interpolation space
\[
X_{\theta}=(X_0^{n-1},X_1^{n-1})_{\theta,\infty}\,.
\]
\begin{lemma}Let $U\subset \R^n$ be bounded and open. For $i=1,\dots,n$, let
 $\alpha_i\in (0,1)$ such that
\[
\theta:=\left(\sum_{i=1}^n\alpha_i\right)-(n-1)>0\,.
\]
Additionally, let $u=(u_1,\dots,u_n)\in C^1(U;\R^n)$.
Then

\begin{equation}
\label{eq:13}
\|M(u_1,\dots, u_n)\|_{X_{\theta}} \leq C(n,\alpha_1,\dots,\alpha_n)
  \prod_{i=1}^n\|u_i\|_{C^{0,\alpha_i}(U)}\,,
\end{equation}
Moreover, for $\tilde \theta<\theta$, $M$ extends to a multi-linear operator
$C^{0,\alpha_1}(U)\times\dots\times C^{0,\alpha_n}(U)\to X_{\tilde \theta}$.
\label{lem:detinter}
\end{lemma}
{\bf Notation:} All constants $C$ in the proof below may depend on
$n,\alpha_1,\dots,\alpha_n$ without explicit statement.
\begin{proof}
We use the representation of real interpolation spaces as trace spaces, see
Definition \ref{def:tracespace}. In
particular, we have
$C^{0,\alpha_i}(U)=(C^0(U),C^1(U))_{\alpha_i,\infty}$, and hence we may choose
$v_i\in W^{1,\infty}_{\mathrm{loc}}(\R^+;C^0(U))$ with  $v_i(t)\in
C^1(U)$, $v_i'(t)\in C^0(U)$ for all $t>0$ such that
\[
\begin{split}
  \|t^{1-\alpha_i} v_i(t)\|_{L^\infty(\R^+;C^1(U))}\leq &C \|u_i\|_{C^{0,\alpha_i}(U)}\,,\\
  \|t^{1-\alpha_i} v_i'(t)\|_{L^\infty(\R^+;C^0(U))}\leq &C \|u_i\|_{C^{0,\alpha_i}(U)}\,,
\end{split}
\] and $u_i=\lim_{t\to 0}v_i(t)$. Then we set
\[
w(t)=M(v_1(t),\dots,v_n(t))\,.
\]
By the multi-linearity of $M$, we have 
\[
\begin{split}
  \|t^{\sum_i(1-\alpha_i)}w(t)\|_{X_1^{n-1}}\leq&
  C\prod_{i=1}^n\left(t^{1-\alpha_i}\|v_i(t)\|_{C^1(U)}\right)\\
  \leq &C \prod_{i=1}^n\|u_i\|_{C^{0,\alpha_i}(U)}
\end{split}
\]
and

\begin{equation}
\begin{split}
  \|t^{\sum_i(1-\alpha_i)} w'(t)\|_{X_0^{n-1}}=& \frac1nt^{\sum_i(1-\alpha_i)}\Bigg\|
  \sum_{i=1}^n (-1)^{i+1} v_i'(t) \d v_1(t)\wedge\dots\wedge\widehat{\d
    v_i(t)}\wedge\dots\wedge\d v_n(t)\\
&+\sum_{j\neq i}(-1)^{i+1} v_i(t) \d v_1(t)\wedge\dots\wedge\widehat{\d
    v_i(t)}\wedge\dots\wedge \d v_j'(t)\wedge \dots\wedge\d v_n(t)\Bigg\|_{X_0^{n-1}}\\
 \leq & t^{\sum_i(1-\alpha_i)} \sum_{i=1}^n \left\|(-1)^{i+1} v_i'(t) \d v_1(t)\wedge\dots\wedge\widehat{\d
    v_i(t)}\wedge\dots\wedge\d v_n(t)\right\|_{X_0^{n-1}}
\end{split}\label{eq:30}
\end{equation}
where we  have used  that
\[
\begin{split}
  \|v_i(t) \d v_1(t)\wedge\dots \wedge\widehat{\d
    v_i(t)}\wedge\dots\wedge \d v_j'(t)&
  \wedge\dots\wedge \d v_n(t)\|_{X_0^{n-1}}=\\
  &\| v_j'(t) \d v_1(t)\wedge\dots\wedge \widehat{\d v_i(t)}\wedge\dots\wedge\d
  v_n(t)\|_{X_0^{n-1}}\,,
\end{split}
\]
which in turn is a consequence of 
\[
\begin{split}
  \d\Big(v_i(t) \d v_1(t)\wedge\dots \wedge\widehat{\d
    v_i(t)}\wedge\dots\wedge \d v_j'(t)
    &\wedge\dots\wedge \d v_n(t)\Big)\\=&\d \left((-1)^{j+i} v_j'(t) \d
    v_1(t)\wedge\dots\wedge \widehat{\d v_i(t)}\wedge\dots\wedge\d
    v_n(t)\right)\,.
\end{split}
\]
From \eqref{eq:30} we get
\begin{equation*}
\begin{split}
  \|t^{\sum_i(1-\alpha_i)}w'(t)\|_{X_0^{n-1}}
\leq & C \sum_{i=1}^n \left(t^{1-\alpha_i}\|v_i'(t)\|_{C^0(U)}\right)
\prod_{j\neq i}\left(t^{1-\alpha_j}\|v_j(t)\|_{C^1(U)}\right)\\
  \leq &C \prod_{i=1}^n\|u_i\|_{C^{0,\alpha_i}(U)}\,.
\end{split}
\end{equation*}
Hence $w(0)\in (X_0^{n-1},X_1^{n-1})_{1-\sum_i(1-\alpha_i),\infty}=X_\theta$, with
$\|w(0)\|_{X_{\theta}}\leq C \prod_i\|u_i\|_{C^{0,\alpha_i}}$. The estimate
\eqref{eq:13} follows from $w(0)=M(u_1,\dots,u_n)$. 

To prove the statement about the extension, we  choose
$\beta_i<\alpha_i$ with $\sum_i{\beta_i}=n-1+\tilde\theta$. Then we have

\begin{equation}
\label{eq:14}
\|M(u_1,\dots,u_n)\|_{X_{\tilde \theta}}\leq \prod_{i=1}^n\|u_i\|_{C_{0,\beta_i}}
\end{equation}
for $u=(u_1,\dots,u_n)\in C^1(U;\R^n)$. Now every $\tilde u=(\tilde
u_1,\dots,\tilde u_n)\in
C^{0,\alpha_1}\times \dots\times C^{0,\alpha_n}$ can be approximated in
$C^{0,\beta_1}\times \dots \times C^{0,\beta_n}$ by sequences of functions in
$C^1(U;\R^n)$, and hence the existence of a unique extension follows from \eqref{eq:14}.
\end{proof}

\subsection{Integrating distributional  Jacobians over sets with fractal boundary}

The purpose of the construction of the interpolation space $X_\theta=(X_0^{n-1},X_1^{n-1})_{\theta,\infty}$ has been
 to make its elements suitable for integration over fractals
of dimension up to (but not including) $n-1+\theta$. The corresponding
definition will be given in the present subsection. This will be similar to the
constructions in  \cite{MR1119189}.

In the following, let $U\subset \R^n$ be fixed, with $d:=\bd\partial U<n-1+\theta$. Let $W$ be the Whitney
decomposition of $U$, cf.~Section \ref{sec:whitn-decomp-box}. 
Since   $M\in X_\theta$,   there exists $\tilde M(\cdot)\in
W^{1,\infty}_{\mathrm{loc}}(\R^+;X_0^{n-1})$ with $\tilde M(t)\in X_1^{n-1}$,
$\tilde M'(t)\in X_0^{n-1}$ for all $t>0$ such that 
\[
t^{1-\theta}\left(\|\tilde M(t)\|_{X_1^{n-1}}+\|\tilde M'(t)\|_{X_0^{n-1}}\right)\leq
\| M\|_{X_\theta}\quad\text{ for all }t\in\R^+\,,
\]
and 
\[
\lim_{t\to 0}\|M-\tilde M(t)\|_{X_0^{n-1}}=0\,,
\]
see Section \ref{sec:tools-from-interp}.
\begin{definition}
\label{def:fracint}
Assume that $n-1+\theta>d$.  For $M\in X_\theta$ 
let  the integral $\int_U\d M$ be defined by 
\[
\int_U\d M:=\sum_{Q\in W} \int_Q \d \tilde M(\diam Q)+\int_{\partial Q}
(M-\tilde M(\diam
Q))\,,
\]
where  $\tilde M \in
W^{1,\infty}_{\mathrm{loc}}(\R^+;X_0^{n-1})$ is chosen as above.
\end{definition}

\begin{lemma}
\label{lem:intwelldef}
The above definition makes $\int_U \d M$ well defined for $M\in X_\theta$, and the map
\[
M\mapsto\int_U \d M
\]
is continuous on $X_\theta$ with $|\int_U\d M|\leq C(U)\|M\|_{X_\theta}$.
\end{lemma}
\begin{proof}
Let $\tilde M(\cdot)$ as above, and let $Q\in W$. First we estimate
  \[
  \begin{split}
    \left| \int_Q \d \tilde M(\diam Q) \right| \leq &\L^n(Q) \|\tilde M(\diam Q)
    \|_{X^{n-1}_1}\\
    \leq &\L^n(Q) (\diam Q)^{\theta - 1} \|M\|_{X_\theta}\,.
  \end{split}
  \]
To estimate $\int_{\partial Q} (M-\tilde M(\diam Q))$, we first note that
\[
\begin{split}
  \|M-\tilde M(\diam
  Q)\|_{X^{n-1}_0}\leq &\int_0^{\diam Q}\|\tilde M'(t)\|_{X_0^{n-1}}\\
  \leq &C\int_0^{\diam Q}t^{\theta-1}\|M\|_{X_\theta}\\
 \leq & C(\diam Q)^\theta
  \|M\|_{X_\theta}\,.
\end{split}
\]
Hence we get
  \[
  \begin{split}
    \left| \int_{\partial Q} (M-\tilde M(\diam Q)) \right| \leq& \H^{n-1}(Q)
    \|M-\tilde M(\diam
    Q)\|_{X^{n-1}_0}\\
 \leq &C \H^{n-1}(Q)(\diam Q)^\theta \|M\|_{X_\theta} \,.
  \end{split}
  \]
  By Theorem \ref{thm:Nkdb} the number of cubes in $W$ of sidelength
  $2^{-k}$ can be estimated by $C 2^{kd}$, where the constant $C$ may depend on
  the domain $U$, and $d=\bd \partial U$. In this way we obtain
  \[
  \begin{split}
    \left| \int_U \d M \right| \leq &\sum_{Q \in W}\L^n(Q) (\diam Q)^{\theta - 1} \|M\|_{X_\theta}+ \H^{n-1}(Q) (\diam Q)^\theta
    \|M\|_{X_\theta}\\
 \leq C &\sum_{k \in \N} 2^{dk} 2^{-(n-1)k} 2^{-\theta k}
    \|M\|_{X_\theta}\,.
  \end{split}
  \]
By the assumption $d < n-1
 + \theta$
 the infinite sum converges absolutely. This proves 
\[
\left|\int_U\d M\right|\leq C\|M\|_{X_\theta}\,,
\]
and in particular it follows that $\int_U\d M$ does not depend on the choice of
$\tilde M$, which makes the integral well defined. Also,  the continuity of $M\mapsto \int_U\d M$  as a map from
$X_\theta$ to $\R$ follows by linearity.
\end{proof}

We are ready to prove Theorem \ref{thm:main1}. In the proof below, all constants
$C$ may depend on $n,\alpha,p,d$ without explicit statement. 
\begin{proof}[Proof of Theorem \ref{thm:main1}]
In this proof we  assume $p>1$, and  define $p'$ by requiring
$p^{-1}+(p')^{-1}=1$. Note that by assumption, we have $p<n/(n-1)$ and hence $p'>n$.\\
Using the representation of H\"older spaces as trace spaces (cf.~Definition \ref{def:tracespace}),  we 
can choose $ v_i:\R^+\to C^1(U)$, $i\in\{1,\dots,n\}$,  such that
\[
\|t^{1- \alpha}  v_i(t)\|_{
L^\infty(\R^+;C^1(U))}+
\| t^{1-\alpha} v_i'(t)\|_{
L^\infty(\R^+;C^0(U))}\leq C\|u\|_{C^{0,\alpha}(U)}
\]
and 
$ \lim_{t\to 0} v_i(t)= u_i$ in $C^{0}(U)$.

We write $v(t)=(v_1(t),\dots,v_{n}(t))$, and claim that
\begin{equation}
  \begin{split}
    \limsup_{t\to 0}\,\sup&\left\{\int_{\R^n}\deg(v(t),U,y)\varphi(y)\d
      y:\varphi\in
      L^{p'}(\R^n)\cap C^\infty(\R^n),\,\|\varphi\|_{L^{p'}(\R^n)}\leq 1\right\}\\
    <&C\left(\|u\|_{C^{0,\alpha}(U;\R^n)}\right)^{n/p}\,,\label{eq:1}
  \end{split}
\end{equation}
where the constant on the right hand side may depend on $U$.
From the estimate \eqref{eq:1} it follows that $\{\deg(v(t),U,\cdot):t\leq 1\}$  is bounded and hence
precompact in
$L^{p}(\R^n)$. By $v(t)\to u$ in $C^0$, it follows  the pointwise convergence
$\deg(v(t),U,\cdot)\to\deg(u,U,\cdot)$ on $\R^n\setminus u(\partial U)$. By
Lemma \ref{lem:bdrydim}, we have 
$\L^n(u(\partial U))=0$, and hence $\deg(v(t),U,\cdot)\to\deg(u,U,\cdot)$ almost everywhere. In combination with the compactness in $L^p$, it follows $\deg(v(t),U,\cdot)\wto \deg(u,U,\cdot)$
in $L^{p}(\R^n)$. In particular, $\deg(u,U,\cdot)\in L^p$  with
$\|\deg(u,U,\cdot)\|_{L^p}<C\|u\|_{C^{0,\alpha}}^n$.  Since the support of
$\deg(u,U,\cdot)$ is bounded, we also have $\deg(u,U,\cdot)\in L^1(\R^n)$, and
the theorem is proved.
It remains to show \eqref{eq:1}. \\
Let $\varphi\in L^{p'}(\R^n)\cap C^\infty(\R^n)$ with $\|\varphi\|_{L^{p'}}\leq 1$. Let $\zeta\in
W^{2,p'}(\R^n)$ be the  solution of 
\[
\Delta \zeta=\varphi \quad\text{ on }\R^n\,,
\]
and define $\psi\in W^{1,p'}_{\loc}(\R^n;\R^n)$ by
\[
\psi(x)=D \zeta(x)-D\zeta(0)\,.
\]
By standard estimates, we have $\|D\psi\|_{L^{p'}}\leq C
\|\varphi\|_{L^{p'}}$.  Hence by Morrey's inequality, we have
\[
[\psi]_{C^{0,1-n/p'}(U)}\leq C\|\varphi\|_{L^{p'}}\,.
\]
Now since $\psi(0)=0$ we have for any $w\in C^{0,\alpha}(U;\R^n)$
\[
\sup_{x\in U}|\psi\circ w(x)|\leq 
[\psi]_{C^{0,1-n/p'}}\left(\sup_{x\in U}|w(x)|\right)^{1-n/p'}\leq
[\psi]_{C^{0,1-n/p'}}\|w\|^{1-n/p'}_{C^{0,\alpha}}\,.
\]
Furthermore for  $x,y\in U$, we have
\[
\begin{split}
  |\psi\circ w(x)-\psi\circ w(y)|\leq &[\psi]_{C^{0,1-n/p'}}|w(x)-w(y)|^{1-n/p'}\\
  \leq & [\psi]_{C^{0,1-n/p'}}\|w\|_{C^{0,\alpha}}^{1-n/p'}|x-y|^{(1-n/p')\alpha}\,.
\end{split}
\]
Let $\tilde\alpha:=\alpha(1-n/p')$. By the above, we have for all $t>0$,
\begin{equation}
  \begin{split}
    \psi\circ v(t)\in &C^{0,\tilde\alpha}(U;\R^n)\,,\\
 \|\psi\circ
    v(t)\|_{C^{0,\tilde \alpha}(U)}\leq&
    C\|\varphi\|_{L^{p'}}\|v(t)\|^{1-n/p'}_{C^{0,\alpha}(U)}\\
\leq &C\|\varphi\|_{L^{p'}}\|u\|^{1-n/p'}_{C^{0,\alpha}(U)}\,,\label{eq:113}
  \end{split}
\end{equation}
where in the last estimate, we have assumed that $t$ is small enough and used
Lemma \ref{lem:holdtrace}.
Next, for $i=1,\dots,n$, we set $\tilde v_i(t)=\psi_i\circ v(t)$,
and
\[
V^j(t):=\left(v_1(t),\dots,v_{j-1}(t),\tilde v_j(t), v_{j+1}(t),\dots,v_n(t)\right)\,.
\]
For $t>0$, we have
\begin{equation}
\left|\int_{\R^n}\varphi(y)\deg(v(t),U,y)\d y\right|=
\left|\int_U\varphi(v(t)(x))\det D v(t)(x) \d x\right|\,.\label{eq:112}
\end{equation}
Using the relation $\d M(V^j(t))=\det D V^j(t)\d x$, we get
\[
\begin{split}
  \sum_{j=1}^n \d M(V^j(t))|_x=&\sum_{j=1}^n \det DV^j(t)|_x \d x\\
  =&\sum_{j=1}^n \d V^j_1(t)|_x\wedge\dots\wedge \d V^j_n(t)|_x\\
  =&\sum_{j=1}^n \left.\partial_j \psi_j\right|_{v(t)(x)}\d v_1(t)|_x\wedge\dots\wedge \d v_n(t)|_x\\
  =& \varphi(v(t)(x)) \det Dv(t)(x)\d x\,.
\end{split}
\]
Inserting this into \eqref{eq:112}, we get
\[
\left|\int_{\R^n}\varphi(y)\deg(v(t),U,y)\d y\right|=\left|\int_U\sum_{j=1}^n \d
  M(V^j(t))\right|\,.
\]
 We set
\[
\begin{split}
  \bar\theta:=&\theta-d+n-1\\
  =&(1-n/p')\alpha+(n-1)\alpha-d\\
  =& \frac{n\alpha}{p}-d\,.
\end{split}
\]
Note that by  $p<\frac{n\alpha}{d}$, we have $\bar\theta>0$.
Now we apply Lemmas \ref{lem:intwelldef} and  \ref{lem:detinter} to obtain
\[
\begin{split}
  \left|\int_{\R^n}\varphi(y)\deg(v(t),U,y)\d y\right|&\leq
  \sum_j\left|\int_U\d M(V^j(t))\right|\\
\leq & C\sum_j\|M(V^j(t))\|_{X_{\tilde \theta}}\\
  \leq &C \|v(t)\|_{C^{0,\alpha}}^{n/p}\\
  \leq & C \|u\|_{C^{0,\alpha}}^{n/p}
\end{split}
\]
where we have also used Lemma \ref{lem:holdtrace} in the last estimate.
This proves \eqref{eq:1} and hence the theorem.
\end{proof}
\begin{remark}
  By the method of proof we are using, we cannot get the estimate
  $\|deg(u,U,\cdot)\|_{L^1}\leq C\|u\|_{C^{0,\alpha}}^n$, since this would
  require $W^{2,\infty}$ estimates on the solution of $\Delta\zeta=\varphi$ with
  $\|\varphi\|_{L^\infty}\leq 1$, which of course do not hold in general.
\end{remark}

\section{Proof of Theorem \ref{thm:convcor}}
\label{sec:proof-thm2}

\subsection{Proof of Theorem \ref{thm:convcor} (i)}
The  first part of the theorem is just a corollary to Theorem \ref{thm:main1}.

\begin{proof}[Proof of Theorem \ref{thm:convcor} (i)]
Let $p<q<\frac{n\alpha}{d}$.
By Theorem \ref{thm:main1}, we have
\[
\|\deg(u_k,U,\cdot)\|_{L^{q}}\leq C(n,\alpha,q,d) \|u\|_{C^{0,\alpha}}^n\,,
\]
and hence $\deg(u_k,U,\cdot)$ is weakly compact in $L^{q}$. 
In particular, $|\deg(u_k,U,\cdot)|^{p}$ is equi-integrable. To show the strong convergence $\deg(u_k,U,\cdot)\to \deg(u,U,\cdot)$ in $L^p$, it is sufficient to show $\deg(u_k,U,\cdot)\to \deg(u,U,\cdot)$ in measure, i.e., for every $\delta>0$,
\[
\L^n\left(\{y:|\deg(u_k,U,y)-\deg(u,U,y)|>\delta\}\right)\to 0 \text{ as }k\to \infty\,.
\]
Since the Brouwer degree is integer-valued, this is equivalent to
\begin{equation}
\L^n\left(\{y:\deg(u_k,U,y)\neq\deg(u,U,y)\}\right)\to 0 \text{ as }k\to \infty\,.\label{eq:116}
\end{equation}
Indeed, let $\e>0$, and choose $k_0$  large enough that $\sup|u-u_k|<\e/2$ for $k>k_0$. Then
\[
y\not \in \{t u(x)+(1-t)u_k(x):t\in[0,1],\,x\in \partial U\}\quad \text{ for all }y\in \R^n\setminus (u(\partial U))_\e\,, \,k>k_0\,.
\]
By the homotopy invariance of the degree, this implies
\[
\deg(u_k,U,y)=\deg(u,U,y)\quad \text{ for all }y\in \R^n\setminus (u(\partial U))_\e\,, \,k>k_0\,.
\]
The claim \eqref{eq:116} now follows from Lemma \ref{lem:bdrydim}. This proves (i).
\end{proof}

\subsection{Proof of Theorem \ref{thm:convcor} (ii)}

The present section and Section \ref{sec:proof-lemmas-sec4} are devoted to the
proof of Theorem \ref{thm:convcor} (ii). It consists of a rather explicit
construction of an example.\\
The basic idea is that one considers sequences $u_m$ of functions defined on
  a self-similar set of given box dimension $d$ (which is the boundary of some
  open set $U$). As the index $m$ increases, the functions $u_m$ use smaller and
  smaller  scales of the self-similar set $\partial U$ to develop ``loops''. Each of
  these loops increases the degree, and has (locally)  controlled $\tilde\alpha$-H\"older
  semi-norm, where $\tilde \alpha$ is slightly larger than $\alpha$. Thus one constructs a sequence that converges to 0  in $C^{0,
    \alpha}$  for which the $L^p$  norm
  of the  degree diverges. 
 For the reader's convenience, we first outline the strategy of proof in a
 little more detail.
\begin{itemize}
\item[1.] In Lemma \ref{lem:Umconst} and \ref{def:UUm}, we construct the
  self-similar set $\partial U$ and the pre-fractals $\partial U^m$, that will
  be helpful for the definition of the loops at scale $m$. To lift maps defined
  on $\partial U^m$ to $\partial U$, we define certain projection maps (see
  Lemma \ref{lem:Pconst}). 
\item[2.] Then we define ``single loops'' (Lemma  \ref{lem:xirhodef}). 
These are defined on $(n-1)$-dimensional boxes of sidelength one.
Also, we find  collections of disjoint $(n-1)$-dimensional boxes on
$\partial U^m$ of sidelength $r^m$ (Lemma \ref{lem:Qmconst}). We work with
Euclidean motions and rescalings to lift the ``single loop'' to each of these
boxes, such that the resulting map will have controlled $\tilde\alpha$-H\"older
semi-norm (see
Definition \ref{def:xizetarho} and Notation \ref{not:EBim}), where $\tilde \alpha $ is slightly larger that $\alpha$. 
\item[3.] We then use the compact embedding between H\"older spaces to  show
  that these functions converge to 0 in $C^{0,\alpha}$, while we may use Lemma
  \ref{lem:sep} to show that the associated Brouwer degrees diverge in $L^p$.
\end{itemize}

From now on, we assume $\tilde\alpha$ to be fixed such that 
\begin{equation}
\alpha<\tilde\alpha<\frac{np}{d}\,.\label{eq:11}
\end{equation}

We collect some useful notation. Firstly, we set
\[
\begin{split}
  L=&\{(x,0)\in \R^2:x\in[0,1]\}\,,\\
  D=&\{(x,y)\in\R^2:0<x<1,\,0<y< \min(x,1-x)\}\,.
\end{split}
\]

\begin{lemma}
\label{lem:Umconst}
Let $1<\bar d<2$. Then there exist $r>0$, $N\in \N$, and for each $i\in\{1,\dots,N\}$ a similarity $S_i:\R^2\to \R^2$ such that the following properties are fulfilled:
\begin{itemize}
\item[(i)] $\H^1(S_i(L))=r$ for $i=1,\dots,N$.
\item[(ii)]  The union $\cup_{i=1}^NS_i(L)$ is the image of a continuous curve with start point $(0,0)$ and end point $(1,0)$.
\item[(iii)] $Nr^d=1$.
\item[(iv)] For $i,j\in\{1,\dots,N\}$, we have
\[
\begin{split}
  S_{i}(D)&\subset D\,,\\
S_{i}(D)\cap
  S_{j}(D)&=\emptyset\quad \text{ if }i\neq j\,,\\
  S_{i}(\overline D)\cap
  S_{j}(\overline D)&=\emptyset\quad \text{ if }
  |i-j|>1\,.
\end{split}
\]
\item[(v)] $r<\frac12$ and $2r^{1-\alpha}\leq 1$.
\end{itemize}
\end{lemma}

The proof can be found in Section \ref{sec:proof-lemmas-sec4}.\\
\\
For the rest of this section, let $d,n$ be fixed with 
$n-1\leq d<n$. Further, let $\bar d=d-(n-2)$ and fix $N,r$ and a set of
similarities $\mathcal S=\{S_1,\dots,S_N\}$ as in Lemma \ref{lem:Umconst}. In
the following, we are going to use the notation introduced in Section \ref{sec:self-simil-fract}.\\
\\
We now define four (orientation preserving)
Euclidean motions $S^*_1,\dots,S^*_4:\R^2\to\R^2 $ by their actions on $(0,0)$ and $(1,0)$,
\[
\begin{array}{rlrl}
S^*_1(0,0)=&(0,1)\quad &S^*_1(1,0)=&(1,1)\\ 
S^*_2(0,0)=&(1,1)\quad &S^*_2(1,0)=&(1,0)\\ 
S^*_3(0,0)=&(1,0)\quad &S^*_3(1,0)=&(0,0)\\ 
S^*_4(0,0)=&(0,0)\quad &S^*_4(1,0)=&(0,1)\,.
\end{array}
\]
Next, for $i_0\in\{1,\dots, 4\}$ and $i_1,\dots,i_m\in \{1,\dots,N\}$, we introduce  the notation
\[
S_{i_0|i_1,\dots,i_m}=S_{i_0}^*\circ S_{i_1}\circ\dots \circ S_{i_m}\,.
\]
\begin{definition}
\label{def:UUm}
For $m\in\N$, let $\tilde U^m\subset \R^2$ be the bounded open set with boundary

\begin{equation}
\begin{split}
  \partial \tilde U^m:=&\bigcup_{i_0=1}^4S_{i_0}^*S^m(L)\,\\
  =&\bigcup_{\substack{i_0\in\{1,\dots,4\}\\i_1,\dots,i_m\in\{1,\dots,N\}}}S_{i_0|i_1,\dots,i_m}(L)\,.
\end{split}\label{eq:47}
\end{equation}
Further, let $\tilde U\subset \R^2$ be the bounded open set with boundary

\begin{equation}
\partial \tilde U=
  \bigcup_{i_0=1}^4S_{i_0}^*K(\S)\,,\label{eq:48}
  \end{equation}
where $K(\S)$ is the attractor set of $\S$, cf.~Section \ref{sec:self-simil-fract}.
 For $m\in\N$, we define
$U^m=\tilde U^m\times (0,1)^{n-2}$ and moreover, $U=\tilde U\times (0,1)^{n-2}$. 
\end{definition}
For a sketch of $\partial U^0,\partial U^1$ and $\partial U^4$ (with $n=2$, $N=5$) see  Figure \ref{fig:Umsketch}. 
\begin{figure}[h]
\begin{center}
\includegraphics[width=.95\linewidth]{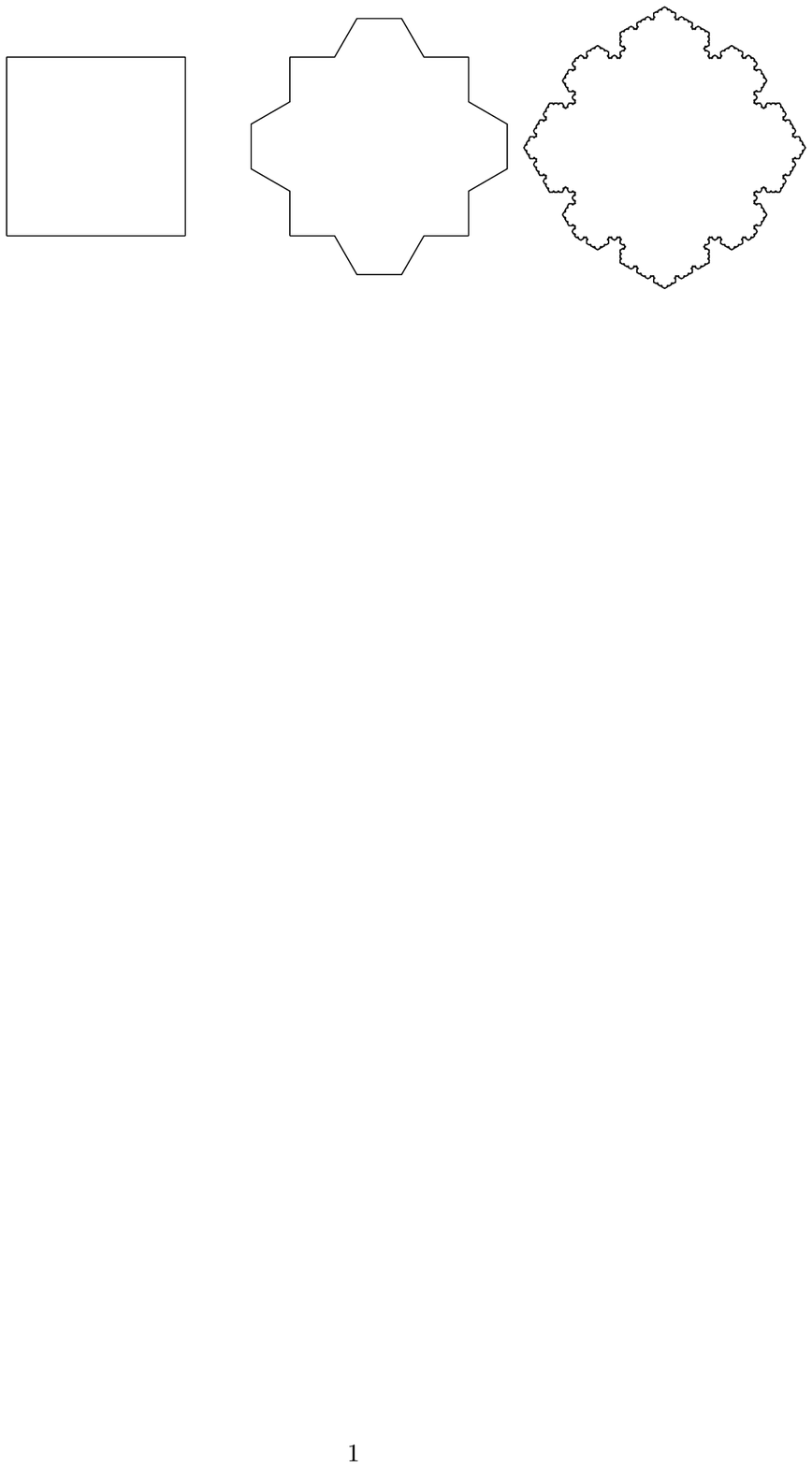}
\caption{The pre-fractals $\partial U^0,\partial U^1$ and $\partial U^4$ (with $n=2$, $N=5$).  \label{fig:Umsketch}}
\end{center}
\end{figure}

\begin{remark}
The sets $\tilde U^m$ and $\tilde  U$ are well defined, since the right hand
sides in \eqref{eq:47}, \eqref{eq:48} are closed  curves by Lemma
\ref{lem:Umconst} (ii) and the definition of the Euclidean motions $S_i^*$, $i=1,\dots,4$.
Also note that by Lemma \ref{lem:Umconst} (iii) and (iv) and Lemma \ref{lem:selfsim} we have $\db \partial U=d$.
\end{remark}

We will need certain ``projection maps'' to pull back maps defined on the
pre-fractals $\partial U^m$ to the fractal $\partial U$:

\begin{lemma}
\label{lem:Pconst}
For every $m\in\N$, there exist  Lipschitz
  maps $P_{m+1}^m:\overline{ U^{m+1}}\setminus U^m\to\partial U^m$ and
  $P^m:\overline{U}\to\overline {U^m}$ with the following properties: 
  \begin{itemize}
  \item $\mathrm{Lip}(P^m_{m+1})\leq 1$, $\mathrm{Lip}(P^m)\leq 1$.
\item If $z=(z',z'')\in\overline{ U^{m+1}}$ with $z'\in\R^2$, $z''\in\R^{n-2}$, then
    $P_{m+1}^m(z)=(\bar z,z'')$ for some $\bar z\in\R^2$.
  \end{itemize}
\end{lemma}
Again, the proof is postponed to Section \ref{sec:proof-lemmas-sec4}.\\
In the statement of the next lemma, we set $B:=B^{(n-1)}(0,1)\times \{0\}$, and
by slight abuse of notation, we write $\partial B:=\left(\partial
  B^{(n-1)}(0,1)\right) \times \{0\}$.
\begin{lemma}
\label{lem:xirhodef}
There exists a Lipschitz map $\tilde \zeta:\overline B\to  \R^n$ (whose
Lipschitz constant only depends on $n$) with the following properties:
\begin{itemize}
\item[(i)] 
$\tilde \zeta(B)\subset \Sb^{n-1}$ and $\tilde \zeta(x)=-e_n$ for all
  $x\in \partial B$.
\item[(ii)]
  If $W\subset \R^n$ is open and bounded with Lipschitz boundary such that
  \[ \overline B\subset \partial W\,
  \]
and the outer  normal to $W$ on $B$ is $e_n$,
  then   $\zeta^{(W)}:\partial W\to \R^n$ defined by 
  \[
  \zeta^{(W)}(x)=\begin{cases}\tilde \zeta(x) &\text{ if }x\in \overline B\\
    -e_n &\text{ else}\end{cases}
  \]
  satisfies
  \[
  \degd(\zeta^{(W)},\partial W ,y)\eq \begin{cases} 1 & \text{ if }y\in B(0,1)\\
    0 & \text{ else. }\end{cases}
  \]
\end{itemize}
\end{lemma}
Again, the proof is postponed to Section \ref{sec:proof-lemmas-sec4}.
\begin{definition}
\label{def:xizetarho}
Let $\rho>0$ and let $W\subset\R^n$ be an open bounded Lipschitz set with
$B_\rho:=B^{n-1}(0,\rho)\times\{0\}\subset \partial W$, such that the  outer normal
of $W$  on $B_\rho$ is $e_n$. Then we define a Lipschitz
map $\zeta_\rho^{(W)}:\partial W\to\R^n$ by
\[
  \zeta_\rho^{(W)}(x)=\begin{cases}\rho^{\tilde \alpha} \tilde \zeta(x/\rho) &
    \text{ if }
    x\in \overline{B_\rho}\\
    -\rho^{\tilde \alpha} e_n & \text{ else, }\end{cases}
\]
where $\tilde \zeta$ has been defined in Lemma \ref{lem:xirhodef}.
\end{definition}

See Figure \ref{fig:ezx} for a sketch of $\zeta_\rho^{(W)}$. From now on, we are going
to drop the superscript $(W)$ for ease of notation, and write $\zeta_\rho^{(W)}\equiv \zeta_\rho$.
\begin{figure}[h]
\begin{center}
\includegraphics[width=.8\linewidth]{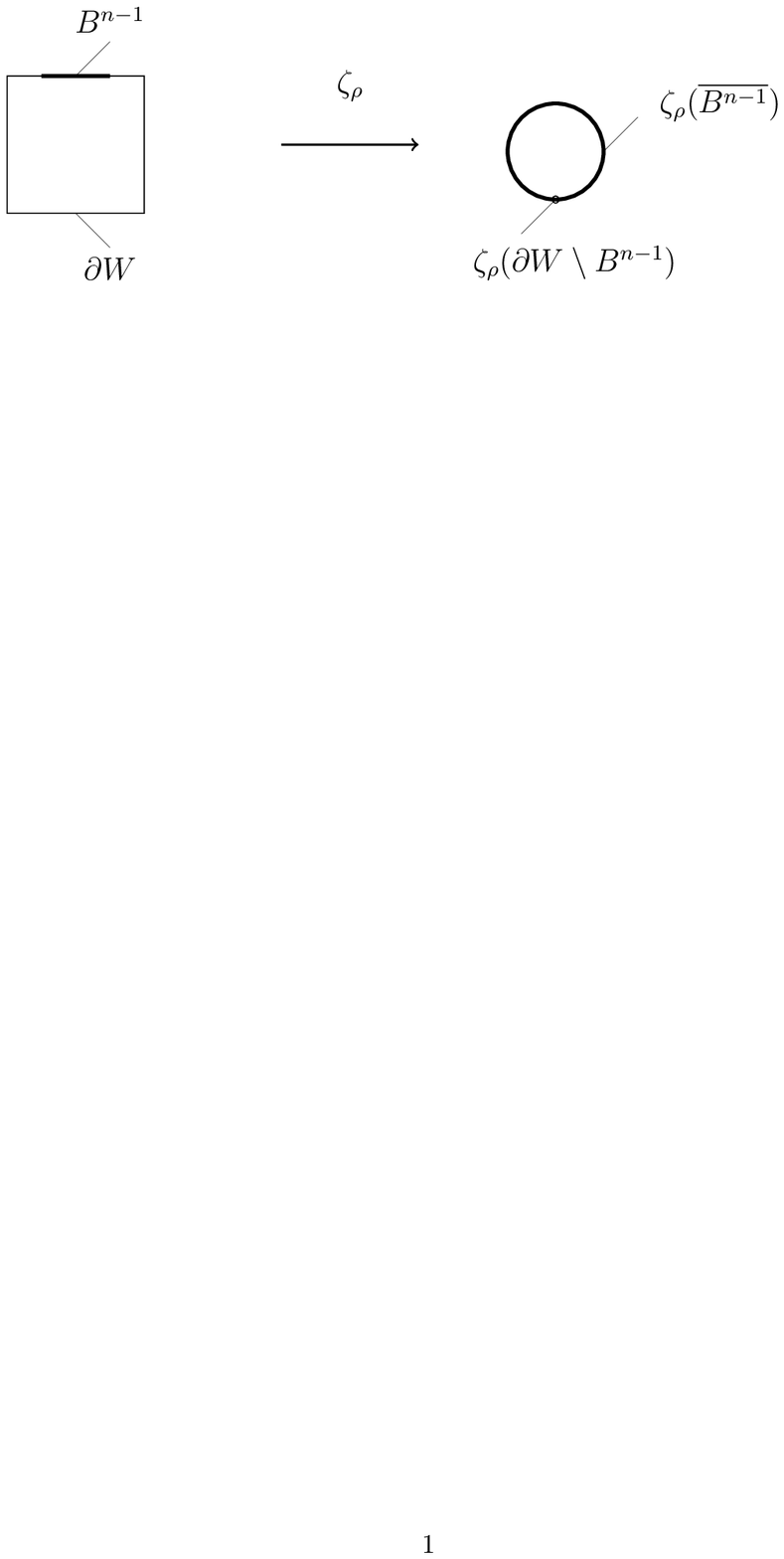}
\caption{A sketch of the construction from Definition
  \ref{def:xizetarho}. In the figure, $B^{n-1}\equiv B^{n-1}(0,\rho/2)\times\{0\}$.  \label{fig:ezx}}
\end{center}
\end{figure}

\begin{remark}
\label{rem:xizetarho}
\begin{itemize}
\item[(i)] As a  consequence of Lemma \ref{lem:xirhodef} (i) and (ii), $\zeta_\rho$ is indeed
  a well defined Lipschitz map with
  \[
  \Lip \zeta_\rho\leq C(n)\rho^{\tilde \alpha-1}\,.
  \]
\item[(ii)] By Lemma \ref{lem:xirhodef} (ii), we have
  \[
  \degd(\zeta_\rho,\partial W,y)\eq \begin{cases}1 & \text{ if }y\in
    B(0,\rho^{\tilde \alpha})\\
    0 & \text{ else. }\end{cases}
  \]

\end{itemize}
\end{remark}

\newcommand{\Q}{{\mathcal Q}}

In the next lemma, by an ``$(n-1)$-dimensional box'', we mean the image of
$[-\rho,\rho]^{n-1}\times\{0\}$ under some Euclidean motion, for some $\rho>0$.
\begin{lemma}
\label{lem:Qmconst}
For every $m\in\N$, there exists a finite collection $\mathcal Q^m$ of
$(n-1)$-dimensional boxes of sidelength
$r^m$,  such that, writing 
\[
\mathcal \Q^m=\bigcup_{i=1}^{\#\Q^m}Q_i^m\,,
\] 
the following holds:
\begin{itemize}
\item[(i)] For every $i\in\{1,\dots,\#\Q^m\}$, we have $Q_i^m\subset\partial
  \tilde U^m\times [0,1]^{n-2}$, with $\tilde U^m$ as in Definition \ref{def:UUm}.
\item[(ii)] If  $i,j\in\{1,\dots,\#\Q^m\}$, $i\neq j$, then $Q_i^m\cap
  Q_j^m=\emptyset$. If additionally $\overline
  {Q_i^m}\cap\overline{Q_j^m}=\emptyset$, then $\dist(Q_i^m,Q_j^m)\geq r^m$.
\item[(iii)] $\lim_{m\to\infty}N^{-m}r^{(n-2)m}\,\#\Q^m=1$\,.
\end{itemize}
\end{lemma}
Again, the proof is postponed to Section \ref{sec:proof-lemmas-sec4}.
\begin{notation}
\label{not:EBim}
  For $m\in \N$ and $x\in\partial U^m$, let $\nu^m(x)$ denote the outward normal
  to $\partial U^m$ at $x$, if it exists. Let $\mathcal N^m\subset \partial U^m$ denote the set of points for which $\nu^m(x)$ does not exist. For $i\in\{1,\dots,\#\Q^m\}$, let $E_i^m$ denote a Euclidean motion that satisfies
\[
\begin{split}
  E_i^m((-r^m/2,r^m/2)^{n-1}\times\{0\})=&Q_i^m\,\\
  E_i^m(e_n)=&E_i^m(0)+\nu^m(E_i^m(0))\,,
\end{split}
\]
and let $B_i^m$ be the $(n-1)$-ball of radius $r^m/2$ at the center of $Q_i^m$, i.e., 
\[
B_i^m:=E_i^m(B^{n-1}(0,r^m/2)\times\{0\})\,.\]
\end{notation}

We are now ready to prove the second part of Theorem \ref{thm:convcor}.

\begin{proof}[Proof of Theorem \ref{thm:convcor} (ii)]
Let $U,U^m\subset \R^n$ be as in Definition \ref{def:UUm},  for
$m\in\N$. 
For $x\in\partial U^m$, we set
\[
v_m(x)=\begin{cases} \zeta_{r^m/2}((E_i^m)^{-1}(x)) &\text{ if } x\in B_i^m\\
-(r^m/2)^{\tilde \alpha } e_n & \text{ if } x\in \partial U^m\setminus \left(\bigcup_{i=1}^{\#\Q^m}B_i^m\right)\,,\end{cases}
\]
where we used the notation introduced in Definition \ref{def:xizetarho} and Notation \ref{not:EBim}.
We immediately see that $v_m$ is Lipschitz with
\begin{equation}
\sup_{x\in\partial U}|v_m|\leq C r^{\tilde \alpha m}\,.\label{eq:2}
\end{equation}
Now let $\alpha<\alpha'<\tilde \alpha$. By \eqref{eq:2}, we have
\begin{equation}
  \label{eq:5}
  \sup_{x\in\partial U}|v_m|\leq \e_m r^{\alpha' m}
\end{equation}
with $\e_m:=C r^{m(\tilde\alpha-\alpha')}$. Next, 
for $|x-y|>r^m$, we have
\begin{equation}
|v_m(x)-v_m(y)|\leq 2\e_m r^{\alpha' m}\leq  2\e_m|x-y|^{\alpha'}\,.\label{eq:3}
\end{equation}
From Remark \ref{rem:xizetarho} (i), we have $\Lip(v_m)\leq Cr^{m(\tilde
  \alpha-1)}$. Hence, for $|x-y|\leq r^m$, we have
\begin{equation}
|v_m(x)-v_m(y)|\leq C|x-y|\e_m r^{m(\alpha'-1)}\leq C\e_m|x-y|^{\alpha'}\,.\label{eq:4}
\end{equation}
By \eqref{eq:5}, \eqref{eq:3} and \eqref{eq:4}, we have
\begin{equation}
\|v_m\|_{C^{0,\alpha'}(\partial U^m)}<C\e_m\to 0\quad \text{ as }m\to\infty\,.\label{eq:6}
\end{equation}

We come to the computation of $\degd(v_m,\partial U^m,\cdot)$. 
To do so, we introduce some additional notation. For $m\in \N$, We  set 
 \[
 B^{*,m}=B(0,(r^m/2)^{\tilde \alpha })
 \]
 and for $i=1,\dots,\#\Q^m$, we define $\zeta_i^m:\partial U^m\to\R^n$ by
\[
\zeta^m_i(x)=\begin{cases} \zeta_{r^m/2}((E_i^m)^{-1}(x)) &\text{ if } x\in B_i^m\\
-(r^m/2)^{\tilde \alpha } e_n &\text{ else.}\end{cases}
\]
We note that by Remark \ref{rem:xizetarho} (ii), we have
\[
\begin{split}
\degd(\zeta_i^m,\partial U^m,x)\eq&\begin{cases}1 & \text{ if } x\in B^{*,m}\\
0 &\text{ else}\end{cases}
\\
=& \chi_{B^{*,m}}\,.
\end{split}
\]
By repeated application of Lemma 
\ref{lem:sep}, with $V:=B_i^m$ for $i\in\{1,\dots,\#\Q^m\}$,
\[
\begin{split}
  \deg(v_m,U^m,\cdot)\eq &\sum_{i=1}^{\# \Q^m}\degd(\zeta_i^m,\partial U^m,\cdot)\\
  \eq &\# \Q^m \,\chi_{B^*}\,,
\end{split}
\]
Hence, 
\[
\|\deg(v_m,U^m,\cdot)\|_{L^p}^p=\left(\# \Q^m\right)^p\omega_n
(r^m/2)^{n\tilde \alpha}\,.
\]
Using  Lemma \ref{lem:Qmconst} (iii) and the relation between $r$ and $N$ from Lemma \ref{lem:Umconst} (iii), we have
\[
\begin{split}
  \lim_{m\to\infty}\|\deg(v_m,U^m,\cdot)\|_{L^p}=&\lim_{m\to\infty}N^m r^{-(n-2)m}
  \omega_n^{1/p}
  (r^m/2)^{n\tilde \alpha/p}\\
  \geq &\lim_{m\to\infty}C(n,p,\tilde \alpha) r^{m(-d+\tilde \alpha n/p)}\\
=&+\infty\,.
\end{split}
\]
We define $ u_m:\partial U\to \R^n$ by
\[
 u_m(x)=v_m(P^m(x))\,,
\]
and extend $ u_m$ from $\partial U$ to $\overline U$ by Theorem \ref{thm:whitney}, such that
\[
\| u_m\|_{C^{0, \alpha'}(U)} <C\|v_m\circ P_m\|_{C^{0,\alpha'}(\partial
  U^m)}\leq C \e_m\to 0\,.
\]
By the compact embedding $C^{0, \alpha'}( U)\to
C^{0,\alpha}( U)$,
\[
u_m\to 0\quad \text{ in } C^{0,\alpha}( U)\,.
\]
Furthermore, note that
\[
\degd(u_m,\partial U,\cdot)\eq \degd(v_m,\partial U^m,\cdot)\quad \text{ for all }m\in\N
\]
and hence $\|\degd(u_m,\partial U,\cdot)\|_{L^p}\to\infty$. This proves the theorem.
\end{proof}

\section{Proof of Lemmas used in Section \ref{sec:proof-thm2}}
\label{sec:proof-lemmas-sec4}

\begin{proof}[Proof of Lemma \ref{lem:Umconst}]
We first construct an auxiliary 
sequence of points in $\R^2$, depending on two parameters $\beta\in(0,\pi/2]$
and $M\in \N\setminus\{0\}$. From this sequence of points, we will construct a generator for a
self-similar fractal later in the proof.\\
\newcommand{\eu}{e_{\mathrm{up}}^\beta}
\newcommand{\ed}{e_{\mathrm{down}}^\beta}
\newcommand{\er}{e_{\mathrm{right}}^\beta}
First, let 
\[
\begin{split}
  \eu:=&(\sin\beta,\cos\beta)\\
  \ed:=&(\sin\beta,-\cos\beta)\\
  \er:=&(1,0)\,.
\end{split}
\]
For $m\in\N$, $m>0$, we define $\lambda_{m,\beta}:\{1,\dots,4m\}\to\R^2$ by
  \begin{equation}
    \lambda_{m,\beta}(l)=\begin{cases}\er &\text{ if }l\in\{1,2m+1\}\\
\eu & \text{ if }1<l<2m+1\\
\ed & \text{ if }2m+1<l\leq 4m\end{cases}\,.
  \end{equation}
The curve one obtains by connecting successively $\sum_{i=1}^l\lambda_{m,\beta}(i)$ for
$l=0,\dots,4m$ is depicted in Figure \ref{fig:lambda}. Note that this curve has
length $4m$. Our next aim is to
concatenate several curves as in Figure \ref{fig:lambda}, letting $m$ increase from 1 to some maximal
value $M$, and then let it decrease to 1 again.
\begin{figure}[h]
\begin{center}
\includegraphics[scale=.7]{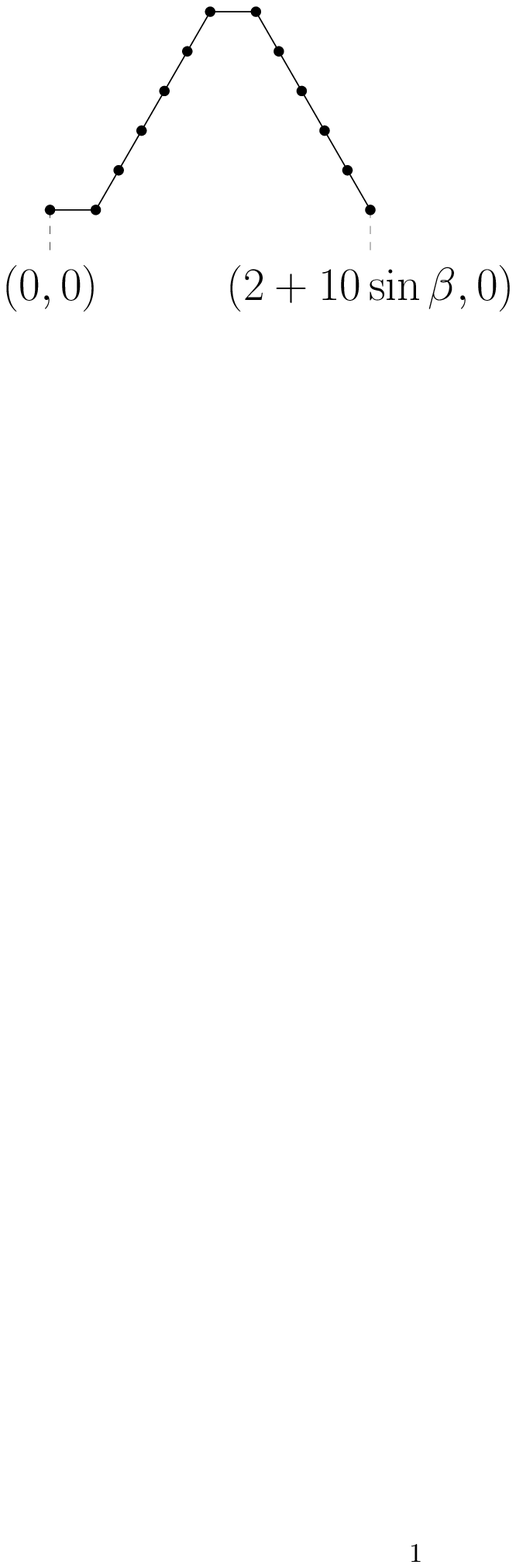}
\caption{The curve constructed from $l\mapsto\sum_{i=1}^l\lambda_{m,\beta}(i)$
  for $m=3$.\label{fig:lambda}}
\end{center}
\end{figure}
For $M\in\N$, $M>0$, we set 
\[
\begin{split}
  a^{(M)}=&(1,2,\dots,M-1,M,M-1,\dots,2,1)\\
  b^{(M)}_l=&4\sum_{j=1}^{l} a_j^{(M)}\quad\text{ for }l=0,\dots,2M-1.
\end{split}
\]
Then for $i=1,\dots,4M^2$, there exists a unique $l_i\in\{1,\dots,2M-1\}$ such
that
\[
b_{l_i-1}^{(M)}< i\leq b_{l_i}^{(M)}\,.
\]
We set
\[
\kappa_{M,\beta}(i):=\lambda_{l_i,\beta}(i-b_{l_i}^{(M)})\,.
\]
Furthermore, we set $\kappa_{M,\beta}(4M^2+1)=\er$, and
for $j=0,\dots,4M^2+1$, we set
\[
\Kappa_{M,\beta}(j):=\sum_{i=1}^j\kappa_{M,\beta}(i)\,.
\]
The curve one obtains by connecting successively $\Kappa_{M,\beta}(j)$ for
$j=0,\dots,4M^2+1$ is depicted in Figure \ref{fig:zeta}. 
\begin{figure}[h]
\begin{center}
\includegraphics[scale=.7]{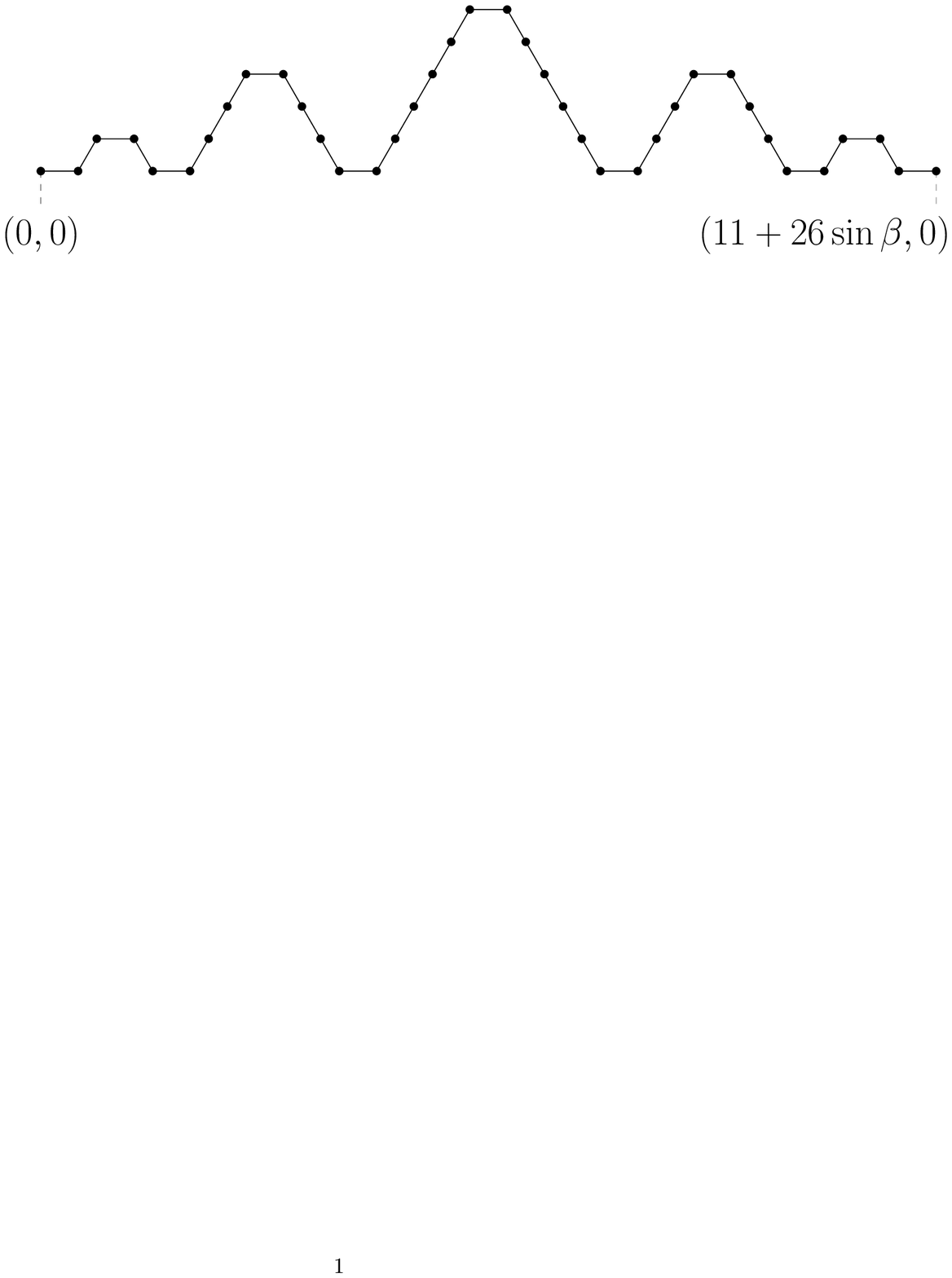}
\caption{The curve constructed from $\Kappa_{3,\beta}$.\label{fig:zeta}}
\end{center}
\end{figure}
We may compute
\begin{equation}
  \begin{split}
    e_1\cdot \Kappa_{M,\beta}(4M^2+1)=&
\sum_{m=1}^M \sum_{i=1}^{4m}e_1\cdot\lambda_{m,\beta}(i)
+\sum_{m=M+1}^{2M-1} \sum_{i=1}^{4(2M-m)}e_1\cdot\lambda_{m,\beta}(i)\\
=&
     2\left(\sum_{m=1}^{M-1}((4m-2)\sin\beta +2)\right)+(4M-2)\sin\beta+3\\
    =& (4M(M-1)+2)\sin\beta+4M-1\,.\label{eq:235}
  \end{split}
\end{equation}

Now set

\begin{equation}
\begin{split}
  \hat N(M):=&4M^2+1\\
  \hat r(M,\beta):=&\left(e_1\cdot \Kappa_{M,\beta}(4M^2+1)\right)^{-1}\\
  =& \left((4M(M-1)+2)\sin\beta+4M-1\right)^{-1}\\
  \hat d(m,\beta):=&-\frac{\log \hat N(M)}{\log\hat r(M,\beta)}\,.
\end{split}\label{eq:46}
\end{equation}

We claim that it is possible to choose $M_0\in\N$ such that 
\begin{equation}
  \begin{split}
    \bar d<&\hat d(M_0,0)\,,\\\label{eq:41}
    (4M_0-1)^{-1}<&\frac12\,,\\
    2(4M_0-1)^{\alpha-1}\leq &1\,.
  \end{split}
\end{equation}
Indeed,
we have $\hat r(M,0)=(4M-1)^{-1}$ and  hence

\begin{equation}
\begin{split}
  \hat d(M,0)=\frac{\log
      (4M^2+1)}{\log (4M-1)}\,.
\end{split}\label{eq:42}
\end{equation}
In particular, note that 
\[
\hat d(M,0)\to 2 \text{ as } M\to \infty\,.
\]
This proves that we may choose $M_0$ such that the first inequality  in
\eqref{eq:41} is fulfilled. After possibly increasing $M_0$, the second and
third hold true too.\\
Since $\hat r(M_0,\cdot)$ is continuous monotone decreasing on $[0,\pi/2]$, the
function $\hat d(M_0,\cdot)$ is continuous monotone decreasing on $[0,\pi/2]$
too. Additionally, we have   $\hat d(M_0,\pi/2)=1$.
Hence, there exists $\beta_0\in (0,\pi/2)$ such that
\[
\bar d=\hat d(M_0,\beta_0)\,.
\]
Now set 
\[
r:=\hat r(M_0,\beta_0)\,,
\quad N:=\hat N(M_0).
\]
We use the following notation: To two points $x\neq y\in \R^2$, we associate the (unique) orientation
preserving Euclidean motion
$S_{x,y}:\R^2\to \R^2$ that maps $(0,0)$ to $x$ and $(1,0)$ to $y$.\\
For $i=0,\dots,N$ write $p(i):=r\Kappa_{M_0,\beta_0}(i)$, and set
\[
S_i:=S_{p(i-1),p(i)}\,.
\]
It remains to verify that $r,N$ and $\mathcal S=\{S_1,\dots,S_N\}$ satisfy the
required properties.\\
Property (i) follows from 
$p(i)-p(i-1)\in
\{re_{\mathrm{right}}^{\beta_0},re_{\mathrm{up}}^{\beta_0},re_{\mathrm{down}}^{\beta_0}\}$
for $i=1,\dots,N$ , and
$|e_{\mathrm{right}}^{\beta_0}|=|e_{\mathrm{up}}^{\beta_0}|=|e_{\mathrm{down}}^{\beta_0}|=1$. Property
(ii) simply follows from $S_1(0,0)=0$, $S_i(1,0)=S_{i+1}(0,0)$ for
$i=1,\dots,N-1$ and $S_N(1,0)=(1,0)$. Property (iii) follows from the definition
of $\hat d(M,\beta)$ in the last line of \eqref{eq:46}, $\bar d=\hat
d(M_0,\beta_0)$, $N=\hat N(M_0)$ and $r=\hat r(M_0,\beta_0)$. Property (iv) is
obvious from an inspection of Figure \ref{fig:fracemp}, where the union of all
$S_i(D)$, $i=1,\dots,N$ is depicted.
\begin{figure}[h]
\begin{center}
\includegraphics[scale=.7]{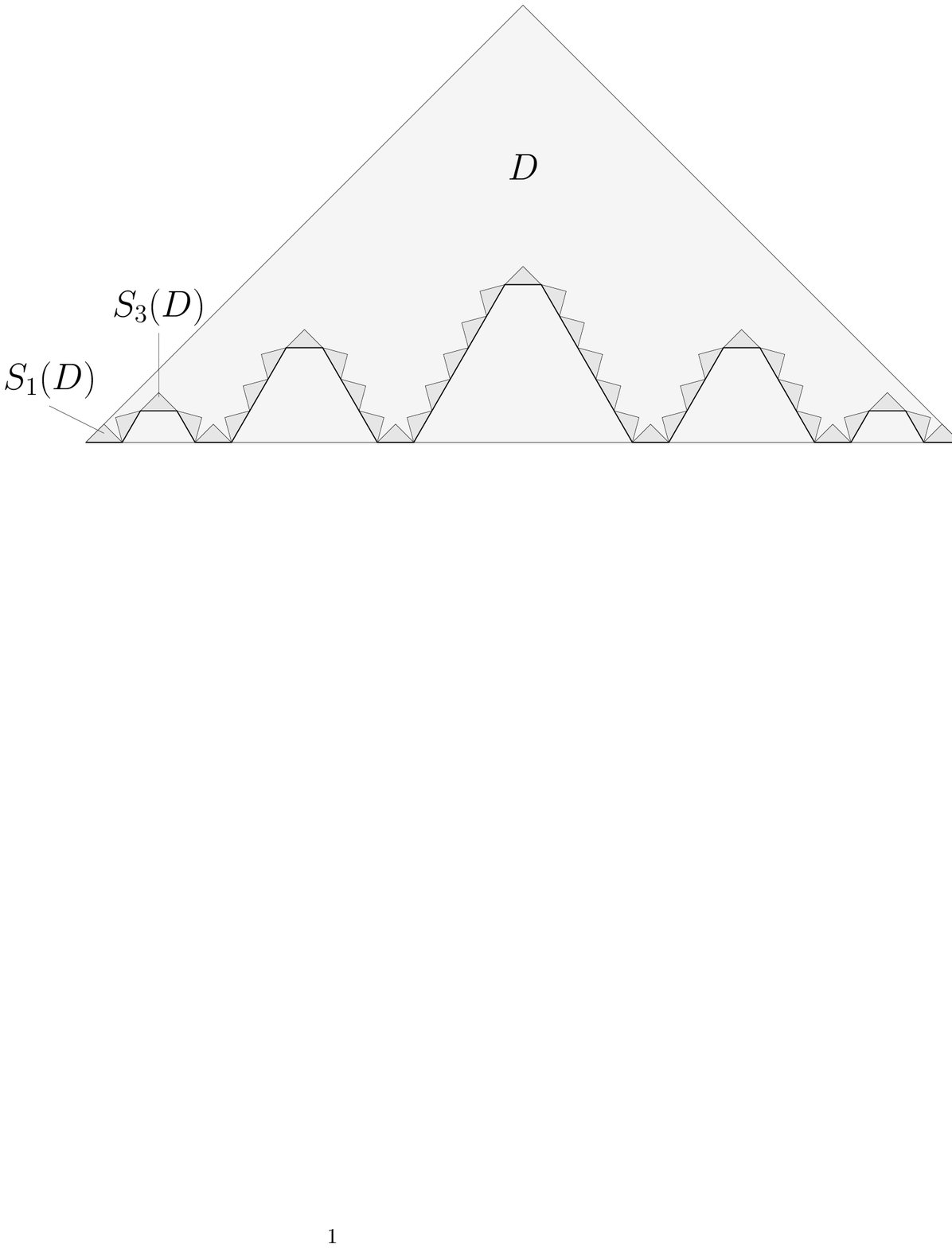}
\caption{$S_i(D)\subset D$ for $i=1,\dots,N$, $S_i(D)\cap S_j(D)=\emptyset$ for $i\neq j$, and $S_i(D)\cap
  S_j(D)=\emptyset$ for $|i- j|>1$. In the picture above, $N=37$.\label{fig:fracemp}}
\end{center}
\end{figure}
Property (v) follows from the second and third line in \eqref{eq:41}, and
$r=\hat r(M_0,\beta_0)\leq (4M_0-1)^{-1}$.
This concludes the proof of the lemma.
\end{proof}

\begin{proof}[Proof of Lemma \ref{lem:Pconst}]
We are going to assume that $n=2$ and construct the maps $P_{m+1}^m$ and $P^m$ for this case only. The  general case follows easily by setting
\[
\begin{split}
  P_{m+1}^m(x)=&( P_{m+1}^{m,(2)}(x_1,x_2),x_3,\dots)\,,\\
  P^m(x)=&( P^{m,(2)}(x_1,x_2),x_3,\dots)\,,
\end{split}
\]
where $P_{m+1}^{m,(2)}$, $ P^{m,(2)}$ denote the maps constructed for $n=2$ below.\\
Let  $A\subset \overline{D}$ be the bounded
closed simply connected set whose boundary contains the union of the curves $L$ and
$S(L)$, see Figure \ref{fig:Afig}.
\begin{figure}[h]
\begin{center}
\includegraphics[scale=.7]{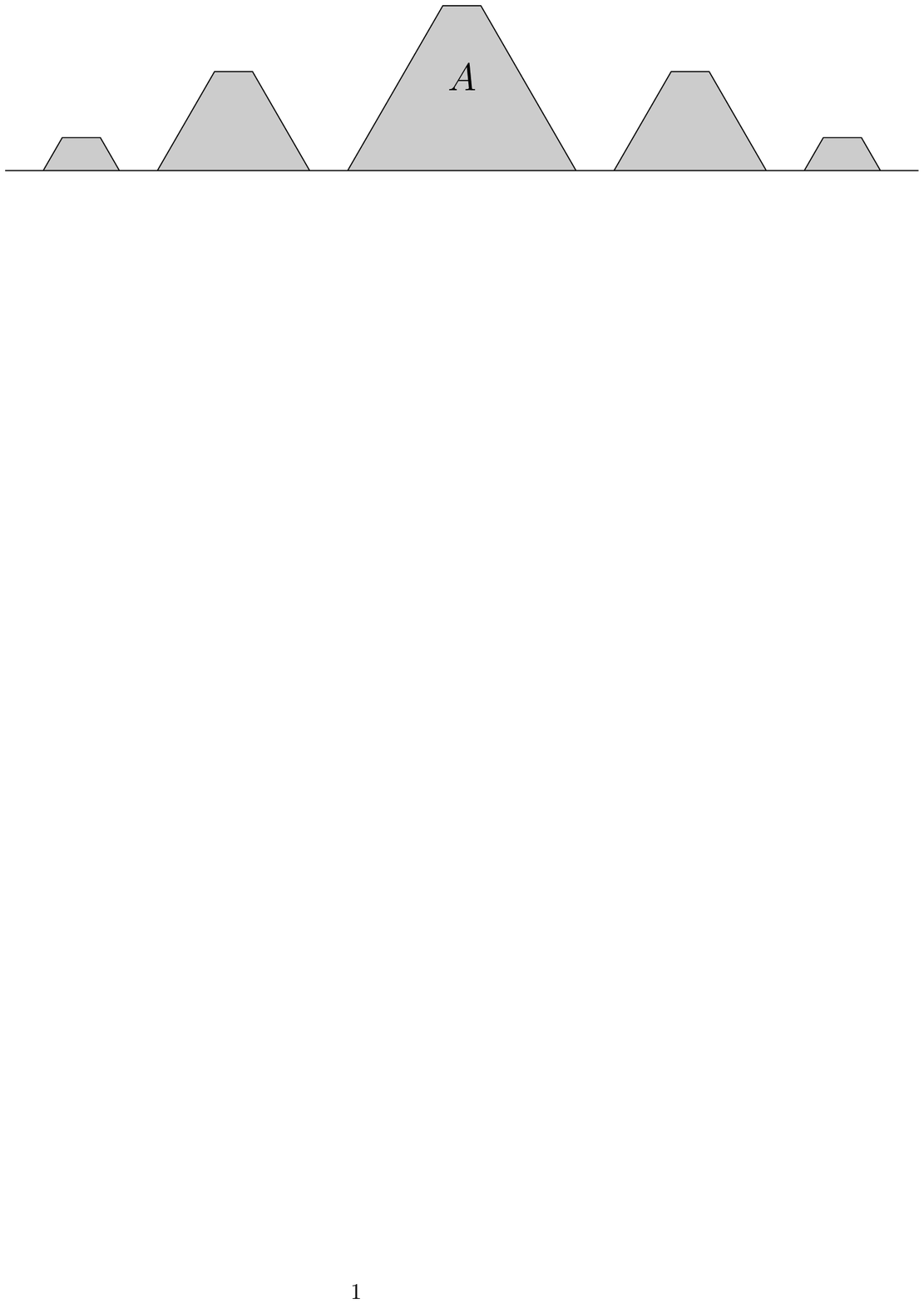}
\caption{The closed set $A$, whose boundary contains $L$ and $S(L)$ (with $N=37$). \label{fig:Afig}}
\end{center}
\end{figure}
I.e., the set $A$ satisfies 
\[
\overline {U^1}\setminus U^0=\bigcup_{i=1}^4 S_{i}^* A\,.
\]
For $x=(x_1,x_2)\in A$, let $P$ be the projection $A\to L$, $x\mapsto
(x_1,0)$. Obviously, $P$ is Lipschitz with Lipschitz constant $\leq 1$.\\
For $m\in \N$, $y\in \overline U^{m+1}\setminus U^m$, there exist $i_0\in\{1,\dots,4\}$,
$i_1,\dots,i_m\in \{1,\dots,N\}$, and $x\in A$ such that
\[
y=S_{i_0|i_1,\dots,i_m}x\,.
\]
We set
\[
P_{m+1}^m(y)=S_{i_0|i_1,\dots,i_m}P(x)\,.
\]
We claim that 
\begin{equation}
  \begin{split}
    P_{m+1}^m:&\overline{U^{m+1}}\setminus U^m\to \partial U^m \quad\text{ is
      well defined,}\\
\label{eq:49} \Lip P_{m+1}^m\leq &1\,.
  \end{split}
\end{equation}
Indeed,  if there exist $i_0,i_0'\in\{1,\dots,4\}$, $i_1,\dots,i_m,i_1',\dots,i_m'\in
\{1,\dots,N\}$, and $x,x'\in A$ with 
$y=S_{i_0|i_1,\dots,i_m}x=S_{i_0'|i_1',\dots,i_m'}x'$, then either $i_j=i_j'$
for $j=0,\dots,m$ and $x=x'$ or $y\in\partial U^m$ and $x,x'\in L$,  in
which case $P(x)=x$, $P(x')=x'$ and hence
$S_{i_0|i_1,\dots,i_m}P(x)=S_{i_0'|i_1',\dots,i_m'}P(x')$. This proves the first
part of \eqref{eq:49}, the second part follows from $\Lip P\leq 1$.
For $l>m\geq 0$, let $ P_l^m:\overline{  U^l}\to\partial  U^m$ be defined by
\[
 P_l^m= P_{m+1}^m\circ\dots\circ  P^{l-1}_l\,.
\]
It is easily seen from this definition and \eqref{eq:49} that $\Lip P_l^m\leq 1$.\\
We come to the definition of $ P^m:\overline { U}\to \overline {
  U^m}$ for $m\in\N$. We set $A^m:=\overline{ U^{m+1}}\setminus U^m$ and note
\[
 U\subset   U^m\cup
\left(\bigcup_{k=m}^\infty   A^k\right)\,.
\]
For $k> m$ and $x\in A^k$, we let $P^m(x)=P_k^m(x)$. Note that this makes $P^m|_U$ well defined with 
\begin{equation}
\Lip P^m|_U\leq 1\label{eq:45}\,,
\end{equation}
since for $k\neq k'$ with $x\in A^k\cap A^{k'}$, we have $P_{k'}^m(x)=P_k^m(x)$. \\
Now let $x\in \partial  U$. There exist $i_0\in\{1,\dots,4\}$ and  a
(possibly non-unique) sequence
$i_k\in\{1,\dots,N\}$, $k=1,2,\dots$, 
such that $x\in S_{i_0|i_1,\dots,i_k}(\overline D)$ for every $k\in \N$
(cf.~\cite{MR3236784}, Chapter 9).
Note that for $k'>k>m$,
\[
 P_{k'}^m\left(S_{i_0|i_1,\dots,i_{k'}}(\overline D)\right)\subset
 P_k^m\left(S_{i_0|i_1,\dots,i_k}(\overline D)\right)\,.
\]
Thus there exists a unique $x'\in S^*_{i_0}L$ (that does not depend on the
choice of the sequence $i_k$) such that
\[
x'\in \bigcap_{k=m+1}^\infty  P_k^m(S_{i_0|i_1,\dots,i_k}(\overline D))\,.
\]
We set $ P^m(x)=x'$. \\
It remains to show that $P^m$ is Lipschitz on $\overline U$ with $\Lip P^m\leq 1$. By \eqref{eq:45}, it is sufficient to show continuity on $\overline U$. Assume we are given $x_j$, $j\in\N$, with $x_j\in U$, $x_j\to x\in \partial U$. We need to show $P^m(x_j)\to P^m(x)$. Indeed,
we will show that for every subsequence, there exists a further subsequence such that convergence holds.
For any subsequence, there has to exist a further subsequence $x_j$ (no relabeling), a sequence $i_1,i_2,\dots\in \{1,\dots, N\}$, and a monotonous increasing function $K:\N\to\N$ with $\lim_{j\to\infty}K(j)=\infty$ such that
\begin{equation}
\label{eq:43}
x_j\in S_{i_0|i_1,i_2,\dots,i_{K(j)}}(\overline D)\,.
\end{equation}
In this case, we must have 
\[
x\in \bigcap_{k=1}^\infty S_{i_0|i_1,i_2,\dots,i_{k}}(\overline D)\,,
\]
and hence 
\begin{equation}
\label{eq:44}
P^m(x)\in \bigcap_{k=m+1}^\infty P^m_k\left( S_{i_0|i_1,i_2,\dots,i_{k}}(\overline D)\right)\,.
\end{equation}
By \eqref{eq:43} and \eqref{eq:44}, we get 
$P^m(x_j)\to P^m(x)$. This proves  the lemma.
\end{proof}


\begin{proof}[Proof of Lemma \ref{lem:xirhodef}]
\newcommand{\xu}{\hat{\underline{x}}}
For $x=(x_1,x_2,\dots,x_{n-1},0)\in B$,  we will identify $x$ with
$(x_1,\dots,x_{n-1})$, and we write $\hat x=x/|x|$. 
We define $\tilde \zeta$ by
\[
\tilde
\zeta(x)=\left(\hat x \sin \pi|x|,\cos \pi|x|\right)\,.
  \]
For $x\in B$, we compute
\begin{equation}
\begin{split}
  D\tilde\zeta(x)=&\left(e_1\otimes e_1+\dots+ e_{n-1}\otimes
    e_{n-1}\right)\frac{\sin \pi|x|}{|x|}+\hat x\otimes\hat x\left(\pi\cos\pi|x|-\frac{\sin\pi|x|}{|x|}\right)\\
  &-\pi\sin \pi|x|\, e_n\otimes \hat x\,.\
\end{split}\label{eq:8}
\end{equation}
From this formula, we see that $\tilde \zeta$ is indeed Lipschitz. All other
properties claimed in  (i) are verified easily.\\ 
For the proof of (ii), first note
that $\zeta^{(W)}:\partial W\to \Sb^{n-1}$ is a well defined Lipschitz map. Furthermore,
$\# (\zeta^{(W)})^{-1}(y)=1$ for all $y\in\Sb^{n-1}\setminus \{-e_n\}$. This
implies that there exists $k\in\{-1,+1\}$ such that
\begin{equation}
\degd(\zeta^{(W)},\partial W,y)=\begin{cases}k&\text{ for all }y\in B(0,1)\\
0&\text{ for all }y\in \R^n\setminus \overline{B(0,1)}\,.\end{cases}\label{eq:7}
\end{equation}
Next, we construct\footnote{The construction of $\lambda$ is not 
  necessary if  we use the relation \eqref{eq:57} -- we do not do so here for the sake of clarity.}
a Lipschitz map $\lambda:\overline{W}\to\R^n$ such that
\[
\lambda=\zeta\quad\text{ on }\partial W,\quad \partial_n \lambda(x)=
\zeta^{(W)}(x)\quad\text{ for all } x\in Q\,.
\]
Such a $\lambda$ exists by (a suitable version of) the Whitney Extension Theorem (see, e.g., Theorem
3.6.2 in \cite{MR1014685}). We are going to compute explicitly the sign of 
\[
\int_{\R^n}\degd(\zeta^{(W)},\partial
W,\cdot)\d\L^n=\int_{\R^n}\deg(\lambda,W,\cdot)\d\L^n\,,
\]
to decide which value for $k$ holds true in \eqref{eq:7}.
In order to do so, we  introduce the following  piece of notation. Let
\newcommand{\ve}{\varepsilon}
\[
\ve_{i_1\dots i_n}=\begin{cases}0 &\text{ if
  }\{i_1,\dots, i_n\}\neq\{1,\dots,n\}\\
\sgn\left((1,\dots,n)\mapsto (i_1,\dots,i_n)\right)& \text{ else.}\end{cases}
\]
In the second line on the right hand side above,
$\mathrm{sgn}\left((1,\dots,n)\mapsto (i_1,\dots,i_n)\right)$ denotes the
signature of the permutation $(1,\dots,n)\mapsto (i_1,\dots,i_n)$. With this
notation, we have for $x\in \overline W$,
\[
\begin{split}
  \det D\lambda(x)=&
  \sum_{i_1,\dots ,i_n=1}^n\ve_{i_1\dots i_n}(\partial_{x_{i_1}}\lambda_{1})\dots(\partial_{x_{i_n}}\lambda_{n})\\
  =&\sum_{i_1,\dots ,i_n=1}^n\partial_{x_{i_n}}\left(\ve_{i_1\dots i_n}(\partial_{x_{i_1}}\lambda_{1})\dots(\partial_{x_{i_{n-1}}}\lambda_{n-1})\lambda_{n}\right)\\
=:& \sum_{i_n=1}^n\partial_{x_{i_n}}f_{i_n}\\
=&\,\div f\,.
\end{split}
\]
Using the formula \eqref{eq:51} for the computation of the degree, and  the Gauss-Green Theorem,  we get
\[
\begin{split}
  \int_{\R^n}\deg(\lambda,W,\cdot)\d\L^n=&\int_{W}\det D\lambda(x)\d x\\
  =&\int_{\partial W} \nu_W\cdot f \d \H^{n-1}\\
  =& \int_{B} e_n\cdot f \d \H^{n-1}\,,
\end{split}
\]
where we denoted the outward normal to
$W$ by $\nu_W$, and used the fact that $f$ vanishes $\H^{n-1}$ almost everywhere
on $\partial W\setminus B$. For $x\in B$, we have
\[
\begin{split}
  f_n (x)=&
  \sum_{i_1,\dots,i_{n-1}=1}^n\ve_{i_1\dots i_{n-1}n}\lambda_{1,i_1}(x)\dots\lambda_{n-1,i_{n-1}}(x)\lambda_{n}(x)\\
  =& \cos \pi|x| \det D(x\mapsto\hat x\sin \pi|x|)\\
=&\pi \cos^2 \pi|x|\left(\frac{\sin \pi|x|}{|x|}\right)^{n-2}\,,
\end{split}
\]
where the value of $(\sin \pi|x|)/|x|$ at 0 is understood to be $\pi$. It follows that
\[
\int_{\R^n}\degd(\zeta^{(W)},\partial
W,\cdot)\d\L^n>0\,,
\]
and hence it follows from \eqref{eq:7} that\footnote{Of course, instead of
  arguing that $k\in\{-1,+1\}$ in \eqref{eq:7} as we did above that equation,
  the value of $k$  could also have been deduced by explicit calculation of $\int_{B} e_n\cdot f \d \H^{n-1}$.}
\[
\degd(\zeta^{(W)},\partial
W,\cdot)\eq \chi_{B(0,1)}\,.
\]
This proves the lemma.
\end{proof}

\begin{proof}[Proof of Lemma \ref{lem:Qmconst}]
For $m\in\N$, choose $k_m\in\N$ such that
\[
k_m^{-1}\geq r^m\geq (k_m+1)^{-1}\,.
\]
This choice implies 
\begin{equation}
  k_mr^{m}\to 1\quad\text{ as }m\to\infty\,.
\label{eq:19}
\end{equation}
Further, let $\tilde{ \mathcal Q}^m$ denote the set of cubes in $\R^{n-2}$ of side length $r^m$
and vertices in $\left([0,1]\cap (\N r^m)\right)^{n-2}$, 
\[
\begin{split}
  \tilde{ \mathcal Q}^m:=&\{(j_1 r^m,(j_1+1)r^m)\times \dots\times (j_{n-2}
  r^m,(j_{n-2}+1)r^m):\\
  &0\leq j_l\leq k_m-1 \text{ for }l=1,\dots,n-2\}
\end{split}
\]
Now let $L'=(0,1)\times \{0\}$, recall the definition of the similarities
$S_{i_0|i_1,\dots,i_m}$ from Section \ref{sec:self-simil-fract},
and set
\[
\begin{split}
  \mathcal Q^m:=\big\{& S_{i_0|i_1,\dots,i_m}(L')\times \tilde Q:\,i_0\in\{1,\dots,4\},\\
  &i_1,\dots,i_m\in\{1,\dots,N^m\}\,,\tilde Q\in \tilde{ \mathcal Q}^m\}\,.
\end{split}
\]
Property (i) follows directly from the definition of $\partial U^m$ (see
Definition \ref{def:UUm}) and $\mathcal
Q^m$. The first part of property (ii) is also obvious; we show the  second
part. Assume $Q_1,Q_2\in\mathcal Q^m$, $Q_1=S_{i_0|i_1,\dots,i_m}(L')\times \tilde
Q_1$, $Q_2=S_{j_0|j_1,\dots,j_m}(L')\times \tilde Q_2$, with
$i_0,j_0\in\{1,\dots,4\}$, $i_1,\dots,i_m,j_1,\dots,j_m\in\{1,\dots,N\}$, and $\tilde
Q_1,\tilde Q_2\in \tilde {\mathcal Q}^m$. 
We must have either 
 $\overline{\tilde Q_1}\cap  \overline {\tilde Q_2}=\emptyset$
or 
$S_{i_0|i_1,\dots,i_m}(L)\cap S_{j_0|j_1,\dots,j_m}(L)=\emptyset$.
In the first case, $\dist(\tilde Q_1,\tilde Q_2)\geq r^m$ and hence also
$\dist(Q_1,Q_2)\geq r^m$. In the second case,
$\dist(S_{i_0|i_1,\dots,i_m}(L), S_{j_0|j_1,\dots,j_m}(L))\geq r^m$ and hence also
$\dist(Q_1,Q_2)\geq r^m$. \\
 Next,  note that by construction, 
$\#\mathcal Q^m=(\# \tilde \Q^m) N^m=(k_m)^{(n-2)}N^m $ and hence property (iii) follows from  \eqref{eq:19}.\\
\end{proof}



\appendix

\section{Properties of H\"older functions}
The following lemma is based on the construction from the well known Whitney
extension Theorem. The proof is mainly a repetition of the proof of Theorem 2.1
in \cite{MR1119189}. However, we could not
find a full proof of the claim we need in the literature, which is why we give
it here.
\begin{lemma}
\label{thm:whitney}
Let $K\subset\R^n$ be compact, $0<\alpha<1$, and $f\in C^{0,\alpha}(K)$. 
Then there exists $\tilde f:\R^n\to \R$ such that
\[
\|\tilde f\|_{C^{0,\alpha}(\R^n)}\leq C(n,\alpha) \|f\|_{C^{0,\alpha}(K)}\,,
\]
and $\tilde f|_K=f$.
\end{lemma}
\begin{proof}
Let $W$ be the Whitney decomposition of $\R^n\setminus K$, see Lemma
\ref{lem:whitneydecomp}. 
For $Q_i\in W$, denote the sidelength of $Q_i$ by $|Q_i|$, and its center by $x_i$.
For each $i\in\N$, fix $p_i\in K$ such that
$\dist(Q_i,p_i)=\dist(Q_i,K)$,   and
let $\tilde Q_i$ be the cube with  center  $x_i$ and $|\tilde
Q_i|=\frac32 |Q_i|$. Fix $\eta\in C^\infty_0(\R^n)$ with $0\leq\eta\leq 1$,
$\eta=1$ on $[-1,1]^n$, $\eta=0$ on $\R^n\setminus [-3/2,3/2]$ and $|D\eta|\leq
4$. Set
\[
\begin{split}
  \eta_i(x):=&\eta\left(\frac{x-x_i}{|Q_i|}\right)\,,\\
    \varphi_i(x):=&\frac{\eta_i(x)}{\sum_{j\in\N}\varphi_j(x)}\,.
  \end{split}
\]
Note that $\{\varphi_i\}_{i\in\N}$ is a partition of unity of $\R^n\setminus K$
subordinate to $\tilde W:=\{\tilde Q_i:i\in\N\}$. We define the extension
$\tilde f$ by
\[
\tilde f(x):=\begin{cases}\sum_{i\in\N}\varphi_i(x)f(p_i)&\text{ for }x\in
  \R^n\setminus K\\f(x) &\text{ for }x\in K\,.\end{cases}
\]
From this definition, we immediately get 
\begin{equation}
\sup_{\R^n}|\tilde f|\leq \|f\|_{C^0(\R^n)}\label{eq:58}\,.
\end{equation}
Obviously, on  $\R^n\setminus K$, $\tilde f$ is a smooth function, and there exists a number
$N=N(n)$ such that for each
$x\in \R^n\setminus K$, 
 there exist at most $N$ pairwise disjoint $Q_i\in W$ such that
 $\varphi_i(x)\neq 0$.\\
Fix $x\in\R^n\setminus K$.
 Let $\mathcal N(x)\subset \N$ denote the index set
defined by $\varphi_i(x)\neq 0$ for $i\in \mathcal N(x)$, and let $i_0\in
\mathcal N(x)$. We compute
\begin{equation}
\begin{split}
  |D\tilde f(x)|=&\left|\sum_{i\in\mathcal N(x)}D\varphi_i(x)f(p_i)\right|\\
  =&\left|\sum_{i\in\mathcal
      N(x)\setminus\{i_0\}}D\varphi_i(x)(f(p_i)-f(p_{i_0}))\right|\\
  \leq &C(n)\sum_{i\in\mathcal N(x)}\frac{[f]_\alpha|p_i-p_{i_0}|^\alpha}{\dist(Q_i,K)}\\
  \leq & C(n) [f]_\alpha \dist(x,K)^{\alpha-1}\,.
\end{split}\label{eq:56}
\end{equation}
With these preparations, we are ready to prove an estimate on $[\tilde
f]_\alpha$. Let $x,y\in \R^n$, and let $x',y'\in \R^n$ with 
\[
|x-x'|=\dist(x,K)\,,\quad |y-y'|=\dist(y,K)\,.
\]
Note that this choice implies that for every point $z$ on the line segment
$[x,x']$, we have $|D\tilde f(z)|\leq C[f]_\alpha|z-x'|^{\alpha-1}$ by
\eqref{eq:56}, and an analogous statement for the line segment $[y,y']$.
Hence 
\[
\begin{split}
  |\tilde f(x)-\tilde f(x')|= & \left|\int_0^{|x-x'|} \frac{\d}{\d
      t}\tilde f\left(x+t\frac{x-x'}{|x-x'|}\right)\d t\right|\\
  \leq & \int_0^{|x-x'|} \left|D\tilde
    f\left(x+t\frac{x-x'}{|x-x'|}\right)\right|\d t\\
  \leq & C(n)[f]_\alpha \int_0^{|x-x'|} t^{\alpha-1}\d t\\
  \leq &C(n,\alpha)[f]_\alpha |x-x'|^\alpha\,.
\end{split}
\]
In the same way, we obtain $|\tilde f(y)-\tilde f(y')|\leq C[f]_\alpha
|y-y'|^\alpha$. \\
 Now assume that $|x-x'|+|y-y'|\leq 4|x-y|$. Then
\[
\begin{split}
  |f(x)-f(y)|\leq &|\tilde f(x)-\tilde f(x')|+|\tilde f(x')-\tilde
  f(y')|+|\tilde f(y')-\tilde f(y)|\\
  \leq & C(n,\alpha)[f]_\alpha(|x-x'|^\alpha+|y-y'|^\alpha+|x'-y'|^\alpha)\\
  \leq & C(n,\alpha)[f]_\alpha|x-y|^\alpha\,.
\end{split}
\]
On the other hand, if $|x-x'|+|y-y'|> 4|x-y|$, then we may assume $|x-x'|>2|x-y|$
and hence the line segment $[x,y]$ is contained in $\R^n\setminus K$, and for
each point $z\in [x,y]$, we have $|D\tilde f(z)|\leq
C[f]_\alpha|x-y|^{\alpha-1}$. Then we get
\[
\begin{split}
  |\tilde f(x)-\tilde f(y)|= & \left|\int_0^{|x-y|} \frac{\d}{\d
      t}\tilde f\left(x+t\frac{x-y}{|x-y|}\right)\d t\right|\\
  \leq & \int_0^{|x-y|}C(n)[f]_\alpha|x-y|^{\alpha-1}\d t\\
  \leq & C(n,\alpha)[f]_\alpha|x-y|^{\alpha}\,.
\end{split}
\]
This proves $[\tilde f]_\alpha\leq C(n,\alpha)[f]_\alpha$, and together with
\eqref{eq:58}, this proves the claim of the present lemma.
\end{proof}

The following lemma is a well known fact from real interpolation theory; we
include it for the non-specialist reader.

\begin{lemma}
\label{lem:holdtrace}
Let $U\subset\R^n$ be open and bounded, 
 $u\in C^{0,\alpha}(U)$, and $v:\R^+\to C^1(U)$ with $v'(t)\in C^0(U)$ for all $t>0$ such that 
 \begin{itemize}
 \item $\|t^{1- \alpha} v_i(t)\|_{ L^\infty(\R^+;C^1(U))}\leq
   C\|u\|_{C^{0,\alpha}(U)}$
 \item $\| t^{1-\alpha} v_i'(t)\|_{ L^\infty(\R^+;C^0(U))}\leq
   C\|u\|_{C^{0,\alpha}(U)}$
 \item $ \lim_{t\to 0} v_i(t)= u_i$ in $C^{0}(U)$\,.
 \end{itemize}
Then
\[
\|v(t)\|_{C^{0,\alpha}(U)}\leq C\|u\|_{C^{0,\alpha}(U)}\text{ for }t\leq
1\,.
\]
\end{lemma}
\begin{proof}
We first observe for $x\in\R^n$, and $t\leq 1$,
\[
\begin{split}
  | v(t)(x)|\leq &\left|u(x)+\int_0^t v'(s)(x)\d s\right|\\
    \leq &\|u\|_{C^0(U)}+\int_0^tC \|u\|_{C^{0,\alpha}(U)}s^{\alpha-1}\d s\\
    \leq & C \|u\|_{C^{0,\alpha}(U)}\,.
  \end{split}
\]
Now let $x,y\in \R^n$ with $x\neq y$.
First assume $|x-y|\leq
t$. Then
\[
\begin{split}
  |v(t)(x)-v(t)(y)|\leq &|x-y|\sup_{\R^n} |Dv(t)|\leq C\|u\|_{C^{0,\alpha}(U)} |x-y|
  t^{\alpha-1}\\
  \leq &C\|u\|_{C^{0,\alpha}(U)} |x-y|^\alpha\,.
\end{split}
\]
On the other hand, if $|x-y|>t$, then
\[
\begin{split}
  |v(t)(x)-v(t)(y)|\leq& |v(t)(x)- u(x)|+| u(x)- u(y)|+|
  u(y)-v(t)(y)|\\
  \leq & C\|u\|_{C^{0,\alpha}(U)} \left(2t^\alpha+|x-y|^\alpha\right)\\
 \leq &C\|u\|_{C^{0,\alpha}(U)} |x-y|^\alpha\,.
\end{split}
\]
This proves $[v(t)]_\alpha\leq C\|u\|_{C^{0,\alpha}(U)} $ and hence the claim of the lemma.
\end{proof}

\bibliographystyle{plain}
\bibliography{rigid}

\end{document}